\documentclass[12pt]{extarticle}

\usepackage{amsmath, amsthm, amssymb, color}
\usepackage[colorlinks=true,linkcolor=blue,urlcolor=blue]{hyperref}
\usepackage{graphicx}
\usepackage{caption}
\usepackage{mathtools}
\usepackage{hyperref}
\usepackage{enumerate}
 \usepackage{comment}

\usepackage{cleveref}
\usepackage{verbatim}
\usepackage{tikz,tikz-cd,tikz-3dplot}
\usepackage{amssymb}
\usetikzlibrary{matrix}
\usetikzlibrary{arrows}
\usepackage{algorithm}
\usepackage[noend]{algpseudocode}
\usepackage{caption}
\usepackage[normalem]{ulem}
\usepackage{subcaption}
\usepackage{multicol}
\oddsidemargin \evensidemargin
\usepackage{makecell}
\usepackage{array}
\usepackage{enumitem}

 \newtheorem{theorem}{Theorem}[section]

\newtheorem{proposition}[theorem]{Proposition}
\newtheorem{lemma}[theorem]{Lemma}
\newtheorem{corollary}[theorem]{Corollary}

\newtheorem{conjecture}[theorem]{Conjecture}

\theoremstyle{definition}

\newtheorem{remark}[theorem]{Remark}

\newtheorem{example}[theorem]{Example}

\newcommand{\PP}{\mathbb{P}}
\newcommand{\RR}{\mathbb{R}}

\newcommand{\CC}{\mathbb{C}}
\newcommand{\ZZ}{\mathbb{Z}}
\newcommand{\NN}{\mathbb{N}}

\title{\bf Varieties of Lines in 3-Space}

\author{Benjamin Hollering, Elia Mazzucchelli, \\
Matteo Parisi, and Bernd Sturmfels}

\date{}
\begin{document}

\maketitle

\begin{abstract} \noindent
We consider configurations of lines in 3-space with incidences prescribed by a graph. This defines a subvariety in a product of Grassmannians.
Leveraging a connection with rigidity theory in the plane, for any graph, we determine the dimension of the incidence variety and characterize when it is irreducible or a complete intersection. We study its multidegree
and the family of Schubert problems
it encodes.
Our
spanning-tree coordinates enable efficient symbolic computations.
We also provide numerical irreducible decompositions for incidence varieties with up to eight lines. These constructions 
with lines play a key role in the Landau analysis of scattering amplitudes in particle physics.
\end{abstract}

\section{Introduction}
\label{sec:Introduction}

Lines in complex projective $3$-space $\PP^3$ are given by points
$A = (a_{12}:a_{13}: a_{14} : a_{23} : a_{24} : a_{34})$
in the {\em Grassmannian} ${\rm Gr}(2,4)$, which is the
hypersurface in $\PP^5$ defined by the Pl\"ucker quadric
\begin{equation}
\label{eq:PluckerQuadric} a_{12}  a_{34} - a_{13} a_{24} + a_{14} a_{23} \,\, = \,\, 0 \, .  \end{equation}
The {\em Pl\"ucker coordinates} $a_{ij}$ are the $2 \times 2$ minors of any
$2 \times 4$ matrix ${\bf A}$ whose rows span the line.
If $B$ is a second line,
also given by a $ 2 \times 4$ matrix ${\bf B}$,
then we consider the bilinear form
\begin{equation}
\label{eq:ABquadric}
 AB \,\, := \,\, {\rm det} \begin{small} \begin{pmatrix} {\bf A} \\ {\bf B} \end{pmatrix} \end{small}
 \,\, = \,\,a_{12} b_{34} - a_{13} b_{24} + a_{14} b_{23} + a_{23} b_{14}
- a_{24} b_{13} + a_{34} b_{12} \, .  \end{equation}
This expression is zero if and only if lines $A$ and $B$ intersect in $\PP^3$.
Note that (\ref{eq:PluckerQuadric}) is $ AA  = 0$.

A configuration of $\ell$ lines in $\PP^3$ is a point 
 in the product  ${\rm Gr}(2,4)^\ell$. In this
 paper we study subvarieties of ${\rm Gr}(2,4)^\ell$
 that are defined by the equations (\ref{eq:ABquadric})
 for prescribed pairs of~lines.

\begin{example}[$\ell=3$]
\label{ex:three}
The $12$-dimensional variety ${\rm Gr}(2,4)^3 \subset (\PP^5)^3$ encodes triples of lines. 
We work in the polynomial ring $\CC[A,B,C]$
in the $18$ variables $a_{ij},b_{ij},c_{ij}$
modulo the Pl\"ucker ideal  $\langle AA, BB, CC \rangle$.
In this quotient ring, the ideal $\langle AB, BC \rangle$ is prime.
Its variety  parametrizes $3$-chains of lines in $\PP^3$. On the other hand,
the ideal $\langle AB, AC, BC \rangle$ is radical but not prime.
It is the intersection of two prime ideals $I_{[3]}$ and $I_{[3]}^*$.
Here, $I_{[3]}$ represents triples of concurrent lines,
while $I_{[3]}^*$ represents triples of coplanar lines.
Each of these two associated primes has ten additional generators of degree three.
These ten cubics are written in \cite[Theorem 3.1]{PST} as $3 \times 3$ determinants, where the entries are linear forms in $A,B,C$.
\end{example}

To model the general case, we abbreviate $[\ell] = \{1,2,\ldots,\ell\}$, and we
fix the polynomial ring $\CC[A_1,A_2,\ldots,A_\ell]$ in $6\ell$ Pl\"ucker coordinates.
We work in its quotient
modulo the prime ideal $\langle A_i A_i : i \in [\ell] \rangle$ that defines
${\rm Gr}(2,4)^\ell$. For any graph $G \subseteq \binom{[\ell]}{2}$, 
we consider the ideal
\begin{equation}
\label{eq:idealI_G}
I_G \quad = \quad \bigl\langle\, A_i A_j \,: \,ij \in G \,\bigr\rangle \, . 
\end{equation}
This defines the {\em incidence variety}  $V_G \subseteq {\rm Gr}(2,4)^\ell$.
This ideal $I_G$ need not be radical; see (\ref{eq:IK4}).
Throughout this paper, $G$ is a graph which is undirected, simple and connected, identified with its set of edges.
If $G$ has two distinct connected components $G_1$ and $G_2$, then $V_G = V_{G_1} \times V_{G_2}$.  
Among the points  in $V_G$, we are          also interested in
those that satisfy $A_i A_j \not= 0$ for all non-edges $ij$ of $G$.
 The Zariski closure of this set is the {\em realization} $W_G$ of $G$.
While $V_G$ has dimension at least $2\ell+3$, because it contains
$V_{K_\ell}$, the realization $W_G$ may be empty. 

\smallskip

We now discuss the organization and main results of this paper.
In Section \ref{sec2} we study the ideals $I_G$
 for graphs $G$  with $\ell \leq 5$ vertices.
We then introduce affine coordinates that identify  ${\rm Gr}(2,4)^\ell$
with $\CC^{4\ell}$. Proposition \ref{prop:WeFind} shows
that passing to $\CC^{4 \ell}$ retains the primary decomposition of $I_G$.
Section \ref{sec3} is devoted to the multidegree $[V_G]$.
This is the cohomology class of $V_G$ in its Pl\"ucker
embedding in $(\PP^5)^\ell$. The multidegree encodes
 a range of Schubert problems 
(Theorem \ref{thm:multidegree})
that are labeled by auxiliary graphs $G_u$.
Corollary \ref{cor:CI} gives a formula for
complete intersections, and Corollary \ref{cor:escobar}
revisits the connection to computer vision in~\cite{EK, PST}.
In Section \ref{sec5} we extend the framework, due to
Elekes-Sharir \cite{ES} and Raz \cite{Raz} in 
computational geometry, which relates incidences of
lines in $\PP^3$ to rigidity theory in the plane~\cite{MS, SS}.
A key role is played by the rigidity matroid and
the Geiringer-Laman Theorem.
The codimension of our variety $V_G$
is at most the rank of $G$ in the rigidity matroid
(Corollary \ref{cor:VG_CI_then_G_ind}). Theorems \ref{thm:wheel} and \ref{thm:triangulations}
present detailed analyses for wheel graphs resp.~polygon triangulations.
These are circuits resp.~bases in the rigidity matroid.

Section \ref{sec7} is the heart of this paper. 
Theorem \ref{thm:codim} gives a formula for the codimension
of the incidence variety $V_G$ of an arbitrary graph $G$.
We characterize graphs $G$ for which $V_G$ 
is a complete intersection (Theorem \ref{thm:CI}) or irreducible (Theorem \ref{thm:irred}). 
This rests on rigidity theory and a new graph-theoretic notation called
\emph{contraction stability}. 
This notion is characterized for $\ell \leq 9$ by the exclusions of 
 certain induced  subgraphs (such as $K_{2,4}$).
 We also show (in Theorem \ref{thm:isprime}) that the ideal $I_G$ is prime
 whenever its variety $V_G$ is irreducible.
 
In Section \ref{sec4} we introduce an affine coordinate
system that is adapted to a choice of spanning tree in the graph $G$.
For instance, for $\ell = 3$,  we now use only eight variables, given by 
two $2 \times 2$ matrices $X_1$ and $X_2$.
The equations for $V_{K_3}$ in spanning tree coordinates are
$${\rm det}(X_1) \,  \, = \,\, {\rm det}(X_2) \,\, = \,\, {\rm det}(X_1 + X_2) \quad = \quad 0 \, .$$
The effectiveness for symbolic computations is demonstrated
in Examples \ref{ex:compgra2}, \ref{ex:K24} and \ref{ex:trianghex}.

Section \ref{sec6} is our contribution to experimental mathematics.
We undertake a numerical study of the varieties $V_G$ 
for graphs $G$ up to $\ell = 8$. Numerical irreducible
decompositions are computed using {\tt HomotopyContinuation.jl} \cite{BT}.
The results are shown in Tables
\ref{table:all-connected-graphs} and~\ref{table:triangle-free-graphs}.
See Proposition   \ref{prop:winner7} for a graph $G$ with $\ell=7$
where $V_G$ has $58$ irreducible components.
Our census for small graphs is made available on {\tt Zenodo} at the supplementary materials website
\begin{equation}
\label{eq:zenodo}
   \hbox{\url{https://zenodo.org/records/17708048}.}
\end{equation}

The original motivation for this project comes from particle physics.
The varieties $V_G$ and $V_{G_u}$ are essential in the
Landau analysis of  scattering amplitudes that are expressed in
Grassmannian coordinates. This is the topic of
our forthcoming companion paper \cite{HMPS},
which will be aimed at a physics audience. Section \ref{sec9} offers a friendly invitation 
for mathematicians.

\section{First Examples}
\label{sec2}

We work in a polynomial ring $\CC[A_1,\ldots,A_\ell]$ in $6\ell $ Pl\"ucker variables.
This is the homogeneous coordinate ring of the ambient space $(\PP^5)^\ell$.
It contains the ideal $I_G = \langle A_i A_j : ij \in G \rangle$ for any graph $G \subseteq \binom{[\ell]}{2}$.
The incidence variety $V_G$ is the subvariety of 
${\rm Gr}(2,4)^\ell \subset (\PP^5)^\ell$  defined by $I_G$.
The realization $W_G$ is a subvariety of $V_G$.
In this section we compute these ideals and varieties for some special graphs. 
Our aim is to exhibit various phenomena which can occur.

\begin{example}[Complete graphs] \label{ex:complete}
Let $G = K_\ell$ be the complete graph. The variety
$V_G$ represents $\ell$-tuples of pairwise
intersecting lines in $\PP^3$. Note that $V_G$ is the union of two irreducible
varieties $V_{[\ell]}$ and $V_{[\ell]}^*$. These represent
concurrent lines and coplanar lines respectively.
As in Example \ref{ex:three}, we write $I_{[\ell]}$
and $I^*_{[\ell]}$ for the  prime ideals of $V_{[\ell]}$ and $V_{[\ell]}^*$.
The concurrent lines ideal  $I_{[\ell]}$ 
was determined in \cite[Section 3]{PST}.
The ideal $I^*_{[\ell]}$ is obtained from $I_{[\ell]}$
by a linear change of coordinates. Namely, for each line $A$,
we apply the Hodge star duality:
\begin{equation}
\label{eq:hodgestar} \bigl(a_{12}: a_{13}: a_{14}: a_{23}: a_{24}: a_{34} \bigr)
\quad \mapsto \quad \bigl(a_{34}: -a_{24}: a_{23}: a_{14} :- a_{13}: a_{12} \bigr) \, . 
\end{equation}

For $\ell \geq 4$, the ideal $I_{G}$ is not a radical ideal,
but it has $I_{[\ell]} + I^*_{[\ell]}$ as an embedded associated prime.
Specifically, for $\ell=4$, a computation reveals the minimal primary decomposition:
\begin{equation}
\label{eq:IK4}
 I_{K_4} \,\, = \,\, I_{[4]} \,\,\cap\, \, I^*_{[4]}\, \,\cap\, \, \bigl(\,I_{K_4} + (I_{[4]})^2 + (I^*_{[4]})^2 \,\bigr) \, . 
 \end{equation}
The prime ideals $I_{[\ell]}$ and $I^*_{[\ell]}$ and their varieties serve as basic building blocks in this paper.
\end{example}

In addition to $K_4$, there are five connected graphs on $\ell=4$ vertices.
Three of these graphs $G$ contain no triangle. Their ideals
$I_G$ are prime and complete intersections. In particular,
$V_G = W_G$ is irreducible.
If $G$ is a triangle with a pendant edge then
$I_G$ is the intersection of two prime ideals. This case is similar to
$I_{K_3} = I_{[3]} \cap I_{[3]}^*$. This leaves one graph to examine.

\begin{figure}[h]
	\centering
\begin{tikzpicture}[scale = .8, every node/.style={font=\small}]

    \def\colsep{6} 
    
    \def\dotradius{1.5pt} 
    \def\labelgap{0.25} 
    
    \draw (0,0) rectangle (2,2);
    \fill[gray!30] (0,2) -- (2,2) -- (0,0) -- cycle;
    \foreach \x/\y in {0/0,0/2,2/0,2/2} {
        \fill[black] (\x,\y) circle (\dotradius);
    }
    
    \node at (-\labelgap,2+\labelgap) {2};
    \node at (2+\labelgap,2+\labelgap) {3};
    \node at (2+\labelgap,-\labelgap) {4};
    \node at (-\labelgap,-\labelgap) {1};
    \draw (0,0) -- (2,2);
    
    \node at (1,-1.0) {\(I_{123} + I^*_{134}\)};

    \begin{scope}[shift = {(0, -4)}]
    \draw[black, thick] (0, 0) -- (1.5,2);
    \draw[black, thick] (2, 0) -- (.5,2);
    \draw[black, thick] (0,.5) -- (2, .5);
    \draw[black, thick] (0, 1) -- (2, 1.65);
    
    \node at (1.5, 2.2) {3};
    \node at (.5, 2.2) {1};
    \node at (2.2, .5) {4};
    \node at (2.2, 1.65) {2};
    
    \end{scope}

    \draw (\colsep,0) rectangle (\colsep+2,2);
    \fill[gray!30] (\colsep + 0,0) -- (\colsep+2,0) -- (\colsep+2,2) -- cycle;
    \foreach \x/\y in {0/0,0/2,2/0,2/2} {
        \fill[black] (\colsep+\x,\y) circle (\dotradius);
    }
    
    \node at (\colsep-\labelgap,2+\labelgap) {2};
    \node at (\colsep+2+\labelgap,2+\labelgap) {3};
    \node at (\colsep+2+\labelgap,-\labelgap) {4};
    \node at (\colsep-\labelgap,-\labelgap) {1};
    \draw (\colsep,0) -- (\colsep+2,2);
    
    \node at (\colsep+1,-1.0) {\(I^*_{123} + I_{134}\)};
    
    \begin{scope}[shift = {(\colsep, -4)}]
    \draw[black, thick] (0, 0) -- (1.5,2);
    \draw[black, thick] (2, 0) -- (.5,2);
    \draw[black, thick] (0,.5) -- (2, .5);
    \draw[black, thick] (0, 1) -- (2, 1.65);
    
    \node at (1.5, 2.2) {1};
    \node at (.5, 2.2) {3};
    \node at (2.2, .5) {2};
    \node at (2.2, 1.65) {4};
    
    \end{scope}

    \draw (2*\colsep,0) rectangle (2*\colsep+2,2);
    \fill[gray!30] (2*\colsep,0) rectangle (2*\colsep+2,2);
    \foreach \x/\y in {0/0,0/2,2/0,2/2} {
        \fill[black] (2*\colsep+\x,\y) circle (\dotradius);
    }
    \node at (2*\colsep-\labelgap,2+\labelgap) {2};
    \node at (2*\colsep+2+\labelgap,2+\labelgap) {3};
    \node at (2*\colsep+2+\labelgap,-\labelgap) {4};
    \node at (2*\colsep-\labelgap,-\labelgap) {1};
    \draw (2*\colsep,0) -- (2*\colsep+2,2);
    
    \node at (2*\colsep+1,-1.0) {\(I_{1234}\)};
    
    \begin{scope}[shift = {(2*\colsep, -4)}]
    \draw[black, thick] (0, 0) -- (2,2);
    \draw[black, thick] (2, 0) -- (0,2);
    \draw[black, thick] (0, 1) -- (2, 1);
    \draw[black, thick] (1, 2) -- (1, 0);
    
    \node at (0, 2.2) {1};
    \node at (1, 2.2) {2};
    \node at (2.1, 2.1) {3};
    \node at (2.2, 1.0) {4};
    
    \end{scope}

    \draw (3*\colsep,0) rectangle (3*\colsep+2,2);
    \foreach \x/\y in {0/0,0/2,2/0,2/2} {
        \fill[black] (3*\colsep+\x,\y) circle (\dotradius);
    }
    \node at (3*\colsep-\labelgap,2+\labelgap) {2};
    \node at (3*\colsep+2+\labelgap,2+\labelgap) {3};
    \node at (3*\colsep+2+\labelgap,-\labelgap) {4};
    \node at (3*\colsep-\labelgap,-\labelgap) {1};
    \draw (3*\colsep,0) -- (3*\colsep+2,2);
    
    \node at (3*\colsep+1,-1.0) {\(I^*_{1234}\)};
    
    \begin{scope}[shift = {(3*\colsep, -4)}]
    \draw[black, thick] (0, 2) -- (2,0);
    \draw[black, thick] (1, 2) -- (0,0);
    \draw[black, thick] (0, 0.3) -- (2, 1.5);
    \draw[black, thick] (0,1) -- (2, .2);
    
    \node at (1, 2.2) {2};
    \node at (0, 2.2) {1};
    \node at (2.2, 1.65) {3};
    \node at (2.2, .3) {4};

    \end{scope}

\end{tikzpicture}
\caption{The four irreducible components in the incidence variety
for the graph $G$ in Example \ref{ex:whiteandblack}.
The last row shows the configurations of four lines given by each component.}
\label{figure:fourlines}
\end{figure}
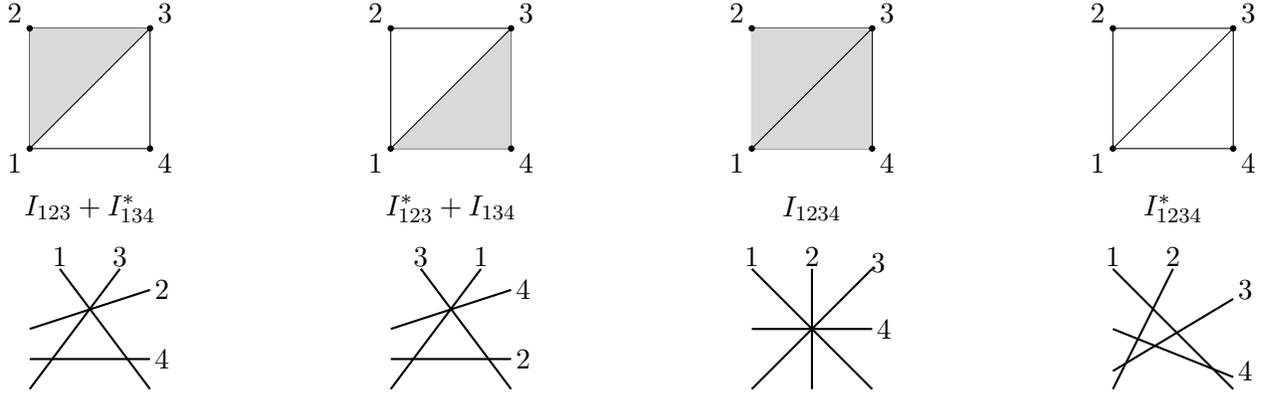

\begin{example}[$\ell = 4$] \label{ex:whiteandblack}
Let $G$ be the complete graph $K_{4}$ with the one edge $\{2,4\}$ removed.
Thus $G$ is the triangulation of a quadrilateral.
The ideal $I_G$ is a radical complete intersection.
Its variety $V_G$ has four irreducible components.
They are illustrated in Figure \ref{figure:fourlines}.
Each of the two triangles is colored black (for concurrent) 
or white (for coplanar). When both triangles have the same color,
we obtain the varieties  $V_{[4]}$ or $V_{[4]}^*$. When the colors differ,
we obtain $V_{123} \cap V_{134}^*$ or $V_{123}^* \cap V_{134}$. The union of the last two 
components is the realization $W_G$.
\end{example}

We now turn to $\ell = 5$, so we work in a polynomial ring with $30$ variables.
There are $21$ connected graphs on five vertices up to isomorphism.
Precisely six of these  are triangle-free: three trees, the $5$-cycle, the $4$-cycle
with a pendant edge, and the bipartite graph $K_{2,3}$. For each of these six graphs,
$I_G$ is a prime and a complete intersection, and we have $V_G = W_G$. 

The $15$ remaining graphs $G$  contain triangles. Their varieties $V_G  \subset
{\rm Gr}(2,4)^5$ are reducible because triangles
always yield decompositions (by Proposition \ref{prop:bases_red}).
For $10$ of these $15$ graphs, the ideal $I_G$ is a complete
intersection, so we have ${\rm codim}(V_G) = |G|$ and
${\rm degree}(I_G) = 2^{|G|}$.
Here degree refers to the total degree  as computed by {\tt Macaulay2} \cite{M2}.
See also Corollary \ref{cor:subscheme}.

\begin{example}[$\ell=5, |G| = 7$] \label{ex:5a}
We discuss two interesting complete intersection graphs.
The first is the triangulated pentagon $G = \{12,23,34,45,15,\,13,14\} $.
The ideal $I_G$ is radical, and its variety $V_G$ has $8$ irreducible components.
They are obtained by labeling the three triangles white and black,
as in Example \ref{ex:whiteandblack}. See Theorem \ref{thm:triangulations}
for triangulated polygons in general.

The second graph we examine is  $G = \{14, 15, \,24, 25, \,34, 35, \,45\}$.
This is the bipartite graph $K_{2,3}$ with one extra edge $45$.
Now, the complete intersection ideal $I_G$ is not radical. It has $9$ associated primes, all of codimension $7$.
The primary decomposition equals
\begin{equation}
\label{eq:5a2}
\begin{matrix}
I_G & = & Q \,\,\, \cap \,\,\,
I_{12345} \,\,\cap \,\, I_{12345}^* \,\,\,\cap\,\,\,\,
(I_{1245} + I_{345}^*) \, \cap \,( I_{1245}^* + I_{345}) \,\,\,\, \cap \,\, \\ & &  
(I_{1345} + I_{245}^*) \, \cap \,( I_{1345}^* + I_{245}) \,\, \cap \,\,
(I_{2345} + I_{145}^*) \, \cap \,( I_{2345}^* + I_{145}) \, .
\end{matrix}
\end{equation}
The last eight components are as before. Their degrees add up to
$2 \cdot 8\,+\,6 \cdot 16 = 112$.
The ideal $Q$ is primary and it defines the variety $W_G$. 
The realizations in $W_G$ are given by taking the same line
for $4$ and $5$ and arbitrary lines $1,2,3$ that intersect the double line $45$.
This example is remarkable in that the main component has
a non-reduced double structure.
\end{example}

We now turn to the five graphs for which the ideal $I_G$ is not
a complete intersection (CI).

\begin{example}[$\ell=5$, not CI] \label{ex:5b}
Four of the five graphs $G$ contain $K_4$ as an induced subgraph.
The corresponding ideals $I_G$ are not radical. We list them by increasing number of edges:

$|G|=7$: \  $K_4$ with a pendant edge;
$V_G = W_G$ has $2$ components of codim $6$ and degree~$12$.

$|G|=8$: \   $K_4$ with a pendant triangle;
$V_G$ has $4$ components of codim $7$; degrees $8,8,16,16$.

$|G|=9$: \ $K_5$ minus one edge;
$V_G$ has $4$ components; $2$ of codim $7$, from $K_5$ below.

\hfill The realization $W_G$ has 
$2$ components of codim $8$ and degree $18$.

$|G|=10$:    $K_5$ (Example \ref{ex:complete}); 
$V_G = V_{12345} \cup V_{12345}^*$; components have codim~$7$
and degree~$8$.

\smallskip

\noindent
The last remaining graph is the wheel graph
$G = W_5 = \{13,14,23,24,\,15,25,35,45\}$.
Its ideal $I_G$ is radical, and $V_G$ has $8$ components.
Here, the prime decomposition equals
\begin{equation}
\label{eq:5a3}
\begin{matrix}
I_G & = & 
I_{12345} \,\,\cap \,\, I_{12345}^* \,\,\,\cap\,\,\,\,
(I_{135} + I_{145}^* + I_{235}^* + I_{245})  \,\cap
(I_{135} + I_{145}^* + I_{235}^* + I_{245})  \\ & & 
(I_{1235} + I_{1245}^*) \,\cap \, (I_{1235}^* + I_{1245}) \,\, \cap \,\,
(I_{1345} + I_{2345}^*) \,\cap \, (I_{1345}^* + I_{2345}) \, .
\end{matrix}
\end{equation}
We see that $W_G$ has
 eight irreducible components. 
   Theorem~\ref{thm:wheel} offers a general explanation.
 \end{example}

Our primary decompositions 
live in a polynomial ring
$\CC[A,B,C,\ldots]$ in $6\ell$ variables.
Computing them is  challenging.
However, the geometric nature of our problem allows us
to use affine coordinates.
We write our lines as the row spans of the following $2 \times 4$ matrices:
\begin{equation}
\label{eq:ABCDE} {\bf A} \, = \, \begin{pmatrix}
1 & 0 & \alpha_1 & \alpha_2 \\
0 & 1 & \alpha_3 & \alpha_4 \end{pmatrix} , \,\,
 {\bf B} \, = \, \begin{pmatrix}
1 & 0 & \beta_1 & \beta_2 \\
0 & 1 & \beta_3 & \beta_4 \end{pmatrix}, \,\,
{\bf C} \, = \, \begin{pmatrix}
1 & 0 & \gamma_1 & \gamma_2 \\
0 & 1 & \gamma_3 & \gamma_4 \end{pmatrix}, \,
\ldots.
\end{equation}
The Pl\"ucker coordinates $a_{ij}, b_{ij},c_{ij} ,\ldots$ are the
$2 \times 2$ minors of these matrices.
The dehomogenization in (\ref{eq:ABCDE}) gives a map from 
$\CC[A,B,C,\ldots]$ onto the polynomial ring
$\CC[\alpha,\beta,\gamma, \ldots]$ in $4 \ell$ variables.
Given any polynomial or ideal in $\CC[A,B,C,\ldots]$,
we denote its image under dehomogenization
with a tilde.  The incidence condition for 
the lines $A$ and $B$
is now given~by 
\begin{equation}
\label{eq:AdotB}
\widetilde{AB} \,\,  =\,\, {\rm det} \begin{small}
\begin{pmatrix} {\bf A} \\ {\bf B} \end{pmatrix} \end{small} \,\, = \,\, 
{\rm det} \begin{pmatrix}
\alpha_1 - \beta_1  & \alpha_2 - \beta_2 \\
\alpha_3 - \beta_3  & \alpha_4 - \beta_4 
\end{pmatrix}.
\end{equation}

\begin{example}[$\ell=3$] \label{ex:12affine}
Three lines have $12$ affine coordinates. The ideal for $K_3$ now equals
$$ \begin{matrix} \widetilde I_{K_3} \,\, = \,\,
\langle \widetilde{AB},\widetilde{AC} ,\widetilde{BC} \rangle & = & \quad
\biggl\langle \hbox{$2 \times 2$-minors of} \ \begin{small}
\begin{pmatrix}
\alpha_1 - \beta_1  & \alpha_2 - \beta_2 & \alpha_1 - \gamma_1  & \alpha_2 - \gamma_2 \\
\alpha_3 - \beta_3  & \alpha_4 - \beta_4 & \alpha_3 - \gamma_3  & \alpha_4 - \gamma_4   
\end{pmatrix}\end{small} \biggr\rangle \smallskip \\ & &  \!  \cap \,\,\,
\biggl\langle \hbox{$2 \times 2$-minors of} \ \begin{small}
\begin{pmatrix}
\alpha_1 - \beta_1  & \alpha_3 - \beta_3 & \alpha_1 - \gamma_1  & \alpha_3 - \gamma_3 \\
\alpha_2 - \beta_2  & \alpha_4 - \beta_4 & \alpha_2 - \gamma_2  & \alpha_4 - \gamma_4   
\end{pmatrix} \end{small}
 \biggr\rangle \, .
\end{matrix}
$$
The two components are prime ideals, given by
  $2 \times 4$ matrices with  independent entries.
  The analogous computation for $\ell = 4$ yields
  two $2 \times 6$ matrices, but  with an embedded prime.
\end{example}

\begin{proposition} \label{prop:WeFind}
The primary decomposition of $I_G$ is obtained
by homogenizing that of $\widetilde I_G$.
\end{proposition}

\begin{proof}[Proof and Discussion]
In our setting, homogenization is the following process.
For each line $A$, we write the equations that relate 
Pl\"ucker coordinates and affine coordinates as follows:
\begin{equation}
\label{eq:aAlines}
a_{23}+\alpha_1 a_{12},\,
a_{24}+\alpha_2 a_{12},\,
a_{13} - \alpha_3 a_{12},\,
a_{14} - \alpha_4 a_{12},\,\,
a_{34} - \alpha_1 \alpha_4 a_{12} + \alpha_2 \alpha_3 a_{12} \, .
\end{equation}
To homogenize an ideal in $\CC[\alpha,\beta,\gamma,\ldots]$, we add these
$5 \ell$ polynomials to it,
we next saturate by the degree $\ell$ Pl\"ucker monomial $a_{12} b_{12} c_{12} \cdots $,
and we finally eliminate the $4 \ell$ unknowns $\alpha,\beta,\gamma,\ldots$.
The result is a homogeneous ideal in the Pl\"ucker coordinate ring $\CC[A,B,C,\ldots]$.

Conversely, let $J$ be any $\ZZ^\ell$-homogeneous ideal  in $\CC[A,B,C,\ldots]$
that contains the ideal of ${\rm Gr}(2,4)^\ell$. Suppose that
none of the variables $a_{12},b_{12},c_{12},\ldots$ is a zero divisor
modulo~$J$. The dehomogenization $\widetilde J$ is an ideal in
$\CC[\alpha,\beta,\gamma,\ldots]$. If we now homogenize
$\widetilde J$ via (\ref{eq:aAlines}) then we recover the original ideal $J$.
This holds because  $a_{12} b_{12} c_{12} \cdots $
is not a zero divisor modulo~$J$. Morever, $J$ is prime 
(resp.~primary) if and only if $\widetilde J$ is prime (resp.~primary).

The general linear group ${\rm GL}(4)$ acts on the product of Grassmannians ${\rm Gr}(2,4)^\ell$.
The condition for two lines to be incident is invariant under
this action. Hence, for any graph $G \subseteq \binom{[\ell]}{2}$,
the ideal $I_G$ is invariant under ${\rm GL}(4)$.
Since  ${\rm GL}(4)$ is connected,
$I_G$ has an equivariant primary decomposition, by \cite{PT}.
This means that each primary ideal is  ${\rm GL}(4)$ invariant,
and so is the irreducible variety it defines.
No ${\rm GL}(4)$ invariant subvariety  of ${\rm Gr}(2,4)^\ell$
can be contained in  a Schubert divisor, like $\{a_{12} = 0\}$,
because ${\rm GL}(4)$ acts transitively on the Schubert divisors of ${\rm Gr}(2,4)$.
Therefore, no Pl\"ucker coordinate is a zero-divisor modulo $ I_G$.

We can thus dehomogenize $I_G$, compute a primary decomposition of the affine ideal
$\widetilde I_G$, and then homogenize each primary component of $\widetilde I_G$.
This process yields a primary decomposition of $I_G$.
Both primary decompositions can be arranged to be ${\rm GL}(4)$ equivariant.

It is instructive to write the action of ${\rm GL}(2,4)$ in affine coordinates,
given by the $2 \times 4$-matrices  in (\ref{eq:ABCDE}), like
 ${\bf A} = \bigl({\rm Id}_2\,\, \alpha \bigr)$.
 Let $U $ be an invertible $4 \times 4$ matrix with
$2 \times 2$ blocks $u_{11},u_{12},u_{21},u_{22}$. 
The right action of $U$ on ${\bf A}$ is given by the linear-fractional transformation
 \begin{equation}
 \label{eq:LFT}    \alpha  \,\, \mapsto \,\,  (u_{11} + \alpha \,u_{21})^{-1} 
 \cdot (u_{12} + \alpha \,u_{22} ) \, . 
 \end{equation}
 This is the non-commutative matrix product.
 Consider a second line,  given by its $2 \times 4$ matrix 
${\bf B} = \bigl(\,{\rm Id}_2\,\, \beta \bigr)$. Then (\ref{eq:LFT}) leaves the
 determinant of  $\alpha - \beta$ invariant, up to a constant.
\end{proof}

Proposition~\ref{prop:WeFind} ensures that all our
computations can be carried out in affine coordinates.
This makes symbolic computations much faster, and it
also greatly facilities numerical computations.
In Section \ref{sec4} we introduce a modification of the
affine coordinates which is even more advantageous.
For our theoretical discussions, however, we continue to use
Pl\"ucker coordinates 
$A = (a_{12}:a_{13}: a_{14} : a_{23} : a_{24} : a_{34})$
on each of the $\ell$ Grassmannians ${\rm Gr}(2,4)$.
This is especially important in the next section,
where the $\ZZ^\ell$-grading plays a key role.

\section{Multidegrees}
\label{sec3}

The homogeneous coordinate ring of $(\PP^5)^\ell$ is 
 the polynomial ring $\CC[A_1,A_2,\ldots,A_\ell]$, where
 $A_i$ is the vector of Pl\"ucker coordinates for the line $i$.
 This ring has a natural $\ZZ^\ell$-grading. Each of the
 six entries of $A_i$ has degree $e_i$,
 which is the $i$th standard basis vector of $\ZZ^\ell$.
All our ideals describe subschemes of ${\rm Gr}(2,4)^\ell$
and they are homogeneous in this $\ZZ^\ell$-grading.

Let $G$ be a graph with $\ell$ vertices as before, and let $c = {\rm codim}(V_G)$.
The incidence variety $V_G$ is a subvariety of dimension $d := 4\ell - c$ in ${\rm Gr}(2,4)^\ell$.
Fix $u \in \NN^\ell$ with $u_1+\cdots+ u_\ell = d$. We denote by
$G_u$ the graph that is obtained from $G$ by attaching $u_i$ pendant edges
at the vertex $i$, for all $i \in [\ell]$. The augmented graph $G_u$ has $\ell+d$
vertices and $|G|+d$ edges.

We consider the map $\,\psi_u: V_{G_u} \rightarrow {\rm Gr}(2,4)^d\,$
that deletes the $\ell$ original lines.
The fiber of $\psi_u$ over a general point in ${\rm Gr}(2,4)^d$ is given by
$d$ constraints on a $d$-dimensional variety. Each constraint
 imposes a Schubert condition on one of the $\ell $ lines 
represented by $G$. We expect the fiber to consist of finitely many points. 
We refer to these points as the {\em leading singularities} of the graph $G_u$.
This name is a reference to 
 Landau analysis for scattering amplitudes  in particle physics (see~Section \ref{sec9}).
 We define the {\em LS degree} of $G_u$ to be the cardinality of the
 generic fiber of $\psi_u$ if this cardinality is finite. We set it to $0$~otherwise.
In this definition, the cardinality of the fiber is understood with multiplicity; see Remark \ref{rmk:slightabuse}.

 \begin{example}[$\ell=1$] \label{ex:onevertex}
 Let $G$ be the graph with one vertex and no edge.
 Then $c=0$~and $u_1 = d = 4$. The graph $G_u = \{12,13,14,15\}$
is  the star tree with four leaves. The variety $V_{G_u}$
is irreducible of dimension $16$.
Each point is a configuration of four lines, labeled $2,3,4,5$,
that intersect a given line, labeled $1$.  The map $\psi_u$
deletes the line $1$. Its image is  ${\rm Gr}(2,4)^4$.
Each fiber of $\psi_u$ consists of the two lines that are incident
to four given lines in~$\PP^3$.
We conclude that the LS degree of $G_u$ is $2$,
expressing the
most basic fact in Schubert calculus.
\end{example}
 
\begin{example}[$\ell=3$]  \label{ex:K3again}
Let $G$ be the complete graph $K_3$.
Then $c=3$ and $d=9$. The vector $u = (u_1,u_2,u_3)$
satisfies $u_1+u_2+u_3 = 9$. In order for the LS degree
to be positive, each vertex of $G$ must have
at least four neighbors in $G_u$. This implies $ u_1,u_2,u_3 \geq 2$.
Hence, up to permuting vertices, there are only two possibilities, namely
$u =(4,3,2)$ and $u=(3,3,3)$.

Each LS degree is twice the LS degree of the
irreducible component $V_{123}$ or $V_{123}^*$.
We therefore restrict to the concurrent lines variety $V_{123}$,
The LS degree of $G_u$ is twice the number of triples of
concurrent lines where  $u_i$ Schubert conditions are imposed on line $i$.

We now show that half the LS degree of the graph $G_{432}$ is $4 = 2 \times 2$.
Line $1$ is one of the two lines satisfying its
$u_1 = 4$ Schubert conditions.
 Next, line~$2$ is one of the two lines incident to
the line $1$ and $u_2 = 3$ other given lines.
The line
$3$ is now uniquely determined, as it contains the
intersection point of lines $1$ and $2$,
and it meets the other $u_3=3$ lines.

Finally, half the LS degree of $G_{333}$ is $8 = 2 \times 2 \times 2$.
This was derived in \cite[Section IV.C]{BVH}.
The argument goes as follows.
For $i \in \{1,2,3\}$, the $u_i=3$ Schubert conditions 
specify that the line $i$ lies in one of the two rulings of 
 a quadratic surface $Q_i$ in $\PP^3$.
The intersection $Q_1 \cap Q_2 \cap Q_3$ 
consists of eight points, by B\'ezout's Theorem.
Let $p$ be one of these eight points. A unique line $L_i$
from the specific ruling of $Q_i$ passes through $p$.
Hence $(L_1,L_2,L_3)$ is a point in the fiber of $\psi_u$.
Conversely, each point in the fiber is one of these eight points.
 \end{example}

We now fix the graph $G$ but we let the vector $u $ range over $\NN^d$.
We shall organize
the LS degrees $\gamma_u$ of the various graphs $G_u$
into a generating function  in $\ell$ formal variables
$t_1,t_2,\ldots,t_\ell$. 
The cohomology ring of the ambient space $(\PP^5)^\ell$ is the truncated polynomial~ring
$$ H^* \bigl( (\PP^5)^\ell , \,\ZZ \bigr) \,\, = \,\,
\ZZ[t_1,t_2,\ldots,t_\ell] / \langle t_1^6, t_2^6, \ldots, t_\ell^6 \rangle \, . $$

The cohomology class of $V_G$ is a homogeneous polynomial of degree $c+\ell$ of the form
$$ [ V_G ] \,\,\, = \,\,\,
 t_1 t_2\, \cdots\, t_\ell \cdot 
 \sum_{u \in \NN^\ell} \gamma_u \,t_1^{4-u_1} t_2^{4-u_2}\, \cdots\,\,  t_\ell^{4-u_\ell}  \,\,\,= \,\,\,
  \sum_{u \in \NN^\ell} \gamma_u \,t_1^{5-u_1} t_2^{5-u_2}\, \cdots\,\,  t_\ell^{5-u_\ell} \, .
 $$
The prefactor $ t_1 t_2\, \cdots\, t_\ell$ is the class of ${\rm Gr}(2,4)^\ell$ in $(\PP^5)^\ell$.
 Each coefficient $\gamma_u$ is a non-negative integer which is zero unless $u_1 + \cdots + u_\ell = d$
 and $u_i \leq 4$ for all $i$. As is customary in commutative algebra \cite{CCA},
  we call $[V_G]$ the {\em multidegree} of  the incidence variety $V_G$.  It can be
 computed in   {\tt Macaulay2} \cite{M2} by applying the command {\tt multidegree}  to the ideal $I_G$.
The following proposition states that the multidegree is precisely the desired generating function.

 \begin{theorem} \label{thm:multidegree}
  The LS degree of $\,G_u$ equals the coefficient $\gamma_u$ in the multidegree of $\,V_G$.
 \end{theorem}
 
 \begin{proof}
 The coefficient of $\prod_{i=1}^\ell t_i^{5-u_i}$ in the cohomology class $[V_G]$ is the  number
of points obtained by intersecting the variety $V_G$ with a product
 $L_1 \times L_2 \times \cdots \times L_\ell$, where $L_i$ is a general linear space
of dimension  $5- u_i$ in $\PP^5$. We obtain such a general subspace $L_i$ by intersecting
$u_i$ general Schubert hyperplanes in $\PP^5$. Hence the number of intersection
points is precisely the cardinality of the fiber of the map $\psi_u: G_u \rightarrow {\rm Gr}(2,4)^d$
discussed above.
\end{proof}

\begin{corollary} \label{cor:CI}
If $\,V_G$ is a complete intersection, meaning $c = |G|$, then the multidegree~is
\begin{equation}
\label{eq:mutlidegCI}  [V_G] \,\, = \,\, 2^\ell \cdot \prod_{i=1}^\ell t_i \cdot
\prod_{ij \in G }(t_i + t_j) \, . 
\end{equation}
\end{corollary}

\begin{proof}
We refer to the $\ZZ^\ell$-grading on $\CC[A_1,A_2,\ldots,A_\ell]$.
The Pl\"ucker quadric $A_iA_i$ for the $i$th factor of ${\rm Gr}(2,4)^\ell$ has degree $2 e_i$. The bilinear
form $A_i A_j$ has degree $e_i + e_j$ for each $ij \in G$. The $\ell+ |G|$ divisors  they define
intersect transversally. Hence we obtain the class $[V_G]$ by multiplying
their divisor classes $2 t_i$ resp.~$t_i + t_j$ in the cohomology ring
$H^* \bigl( (\PP^5)^\ell , \,\ZZ \bigr)$.
\end{proof}

\begin{example}[$\ell=3$] \label{ex:K3d}
Consider the complete graph $K_3$.
The multidegree formula (\ref{eq:mutlidegCI}) gives
$$ [V_{K_3}] \,= \,  [V_{123}] + [V_{123}^*] \,=\,
\, 8 \cdot t_1 t_2 t_3 (t_1 + t_2)
(t_1 + t_3) (t_2 + t_3) \,\,=\,\,
8  \cdot \biggl(\,\sum_{\pi \in S_3} t_{\pi(1)}^3 t_{\pi(2)}^2 t_{\pi(3)} \biggr)
\,+\, 16 \cdot t_1^2 t_2^2 t_3^2 \, . $$
The coefficients $8$ and $16$ are the LS degrees found in Example \ref{ex:K3again}.
We here compute them using Theorem \ref{thm:multidegree} and Corollary \ref{cor:CI}.
For addivity of the multidegree see \cite[Theorem 8.53]{CCA}.
\end{example}

\begin{remark} \label{rmk:slightabuse}
Our symbol $[V_G]$ is a slight abuse of notation.
Namely, by $V_G$ we here mean the subscheme of $(\PP^5)^\ell$
defined by the ideal $I_G$. Algebraically speaking,
$[V_G]$ is the multidegree of the $\ZZ^\ell$-graded ideal $I_G$. 
The point is that the incidence scheme $V_G$ need not be reduced.
This becomes relevant for the second graph $G$ in Example~\ref{ex:5a}.
Here the multidegree equals
$$ [V_G] \,\,\,=\,\,\, 2^7 \cdot \prod_{i=1}^5  t_i  \cdot \biggl( \,
\prod_{j=1}^3 \,(t_j + t_4)(t_j + t_5) \biggr) \cdot
(t_4 + t_5) \, . $$
This polynomial is the sum of the seven multidegrees for the
irreducible components in (\ref{eq:5a2}).
The summand $[Q]$ is twice the multidegree of
the reduced variety given by the radical of $Q$.
\end{remark}

In Section~\ref{sec2} we reported the
degree for various subschemes in ${\rm Gr}(2,4)^\ell$. What we meant there is
 the total degree of the  ideal in $\CC[A_1,\ldots,A_\ell]$ divided by $2^{\ell}$.
 The factor $2^{\ell}$  is the total degree of ${\rm Gr}(2,4)^\ell$ in $(\PP^5)^\ell$.
In particular,
the degree of $V_G$ equals the total degree  of the ideal $\widetilde I_G$ 
in the affine coordinates given
in Proposition \ref{prop:WeFind}.
We summarize this as follows.

\begin{corollary} \label{cor:subscheme}
The total degree of a subscheme $V$ of ${\rm Gr}(2,4)^\ell$
equals the class $1/2^\ell \cdot [V]$ evaluated at
$t_1 = \cdots = t_\ell = 1$. For $V_G$, this is $1/2^\ell$ times the
sum of all LS degrees $\gamma_u$.
\end{corollary}

A formula for the multidegree of the concurrent lines variety
$V_{[\ell]}$ was stated in \cite[eqn (11)]{PST} and proved in \cite{EK}.
We next show that this result follows easily from Theorem \ref{thm:multidegree}.

\begin{corollary}[Escobar-Knutson] \label{cor:escobar}
The multidegree of the concurrent lines variety equals
$$ [V_{[\ell]}] \, = \,  t_1^3 t_2^3 \cdots t_\ell^3 \,\cdot \,
\biggl( 4 \cdot \sum_{(i,j)} t_i^{-2} t_j^{-1} \, + \,
8 \cdot \! \sum_{\{i,j,k\}} t_i^{-1} t_j^{-1} t_k^{-1} \biggr) \, . $$
\end{corollary}

\begin{proof}
For $\ell = 3$, twice the right hand side is equal to $[V_{K_3}]$ in Example \ref{ex:K3d}.
The coefficients $4$ and $8$ in front of the sums are derived by the argument 
 in Example \ref{ex:K3again}.
Consider one of the $\ell(\ell-1)$ monomials 
$\prod_{i=1}^\ell t_i^{5-u_i}$ in the left sum.
The line $i$ is one of two lines incident to
$u_i=4$ given lines.
 Next, the  line~$j$ is one of two lines incident to
the line $i$ and  $u_j=3$ given lines. All other lines
are uniquely determined.
 In total, we had $\,\gamma_u = 4 = 2 \times 2$ choices.

Next consider one of the $\binom{\ell}{3}$ monomials in the right sum.
We have $3$ Schubert constraints on  lines $i,j,k$
and $2$ constraints on each of the other $\ell-3$ lines. The
argument with B\'ezout's Theorem in Example~\ref{ex:K3d} shows that there we have $8$ choices for the
triple  $i,j,k$. After this, the other $\ell-3$ lines are 
uniquely determined. Hence the coefficient of the monomial is $8$.
No other monomials can appear in $[V_{[\ell]}]$ because
each line needs at least $4$ neighbors in $G_u$.
\end{proof}

\begin{remark}
The multidegree sees only the top-dimensional components of $V_G$.
The same holds for the command {\tt degree} in {\tt Macaulay2} \cite{M2}.
This is relevant for graphs with many lower-dimensional components,
like that in Proposition \ref{prop:winner7}. To make the LS degrees of all components visible,
we must increase $d$ before defining the multidegree and the graphs $G_u$.
\end{remark}

The decompositions seen in (\ref{eq:5a2}) and (\ref{eq:5a3})
are mapped to themselves under the involution given
by the Hodge star in (\ref{eq:hodgestar}).
We close the section with the result that this always holds.

\begin{proposition} \label{prop:hodge} The Hogde star operation acts on the ideals
$I_G$ and their primary decompositions.
Pairs of primary components that are dual to each other have the same multidegree.
\end{proposition}

\begin{proof}
The Hodge star (\ref{eq:hodgestar}) maps
each line in $\PP^3$ to its dual line. The two lines come from
an orthogonal pair of $2$-dimensional linear subspaces in 
$\CC^4$. Incidences are preserved.
Concurrent lines are mapped to coplanar lines and vice versa.
The LS degree of each component of $V_{G_u}$ is preserved
because the projection
$\psi_u:V_{G_u} \rightarrow {\rm Gr}(2,4)^d$ is equivariant under (\ref{eq:hodgestar}).
\end{proof}

\section{The Rigidity Matroid}
\label{sec5}

Configurations of lines in $3$-space
can be studied using techniques from
rigidity theory \cite{MS, SS}.
Here we follow the approach pioneered by
Elekes and Sharir \cite{ES}
and developed by Raz \cite{Raz}.  We identify $\RR^4$ with the space $\RR^2 \times \RR^2$ of pairs
$(s,t)$ of points $s = (s_1,s_2)$ and $t=(t_1,t_2)$ in the 
Euclidean plane $\RR^2$. Writing $(u,v)$ for another pair, we write the lines $A,B$ in~\eqref{eq:ABCDE} as
\begin{equation}
\label{eq:matricesST1} 
 \begin{matrix} \alpha_1 = s_2 - t_2, & \alpha_2 = t_1 - s_1 \, , & \beta_1 = u_2 - v_2 \, , & \beta_2 = v_1 -u_1 \, ,\\
\alpha_3 = s_1 + t_1 \, , & \alpha_4 = s_2 + t_2  \, , & \beta_3 = u_1 + v_1 \, ,  & \beta_4 = u_2 + v_2 \, .
 \end{matrix}
\end{equation}
Then the equation $\widetilde{AB} = 0$ for 
the incidence of lines $A$ and $B$ in (\ref{eq:AdotB}) translates into
\begin{equation}
\label{eq:matricesST2} 
||s - u||^2 \,\, = \,\, (s_1 - u_1)^2 \,+ \,(s_2 - u_2)^2 \quad = \quad
(t_1 - v_1)^2 \,+ \,(t_2 - v_2)^2  \,\, = \,\,
||t - v||^2 \, .
\end{equation}
This derivation establishes the following interpretation for the incidence of two lines in $\PP^3$.

\begin{lemma} \label{lem:pointpairs}
The Grassmannian ${\rm Gr}(2,4)$ is a natural compactification of the space $\RR^2 \times \RR^2$
of planar point pairs. Given two such pairs, the Euclidean distance in the first component
equals that in the second component if and only if the two corresponding lines intersect in~$\PP^3$.
\end{lemma}

It makes sense to consider real numbers when speaking about the Euclidean distance among points in the plane. 
On the other hand, all our varieties will continue to be complex. To study $\ell$ lines, by (\ref{eq:matricesST1}), our polynomial ring $\CC[\alpha,\beta,\gamma, \ldots]$ in $4\ell$ variables is  rewritten as
$\CC\bigl[s^{(i)}_1,s^{(i)}_2,t^{(i)}_1,t^{(i)}_2\,: \, i \in [\ell] \bigr]$.
 The ideal $I_G$ of a graph $G \subseteq \binom{[\ell]}{2}$ is generated by the quadrics
  $$ || s^{(i)} - s^{(j)} ||^2 \,-\, || t^{(i)} - t^{(j)} ||^2 \qquad \hbox{for all edges $ij \in G$}.$$
  The variety of $I_G$ is a cone in $\CC^{4\ell}$, and its closure in
${\rm Gr}(2,4)^\ell$ is the incidence variety $V_G$.
Abusing notation, we now also write $V_G$ for this affine cone.

As a direct consequence of Lemma \ref{lem:pointpairs}, we 
obtain the connection to rigidity theory in dimension $2$.
We write $\RR^{2 \ell}$ for the space of
configurations $s= (s^{(1)}, \ldots, s^{(\ell)})$ of $\ell $ points
 in the Euclidean plane.
The following map records the pairwise distances indexed by the graph~$G$:
\begin{equation}
\label{eq:deltamap} \delta_G \,: \,  \RR^{2\ell}\, \rightarrow \,\RR^G \,\,,\,\,\,s \,\mapsto \,
\bigl(\, || s^{(i)} - s^{(j)} ||^2\, )_{ij \in G} \, . 
\end{equation}
The image and fibers of the map $\delta_G$ are central objects
in rigidity theory \cite{MS, SS}: the dimension of the image of $\delta_G$ is equal to the \emph{rigidity rank} of $G$, as we explain below.
Lemma~\ref{lem:pointpairs} implies the following result, which is what Raz \cite{Raz} calls the {\em Elekes-Sharir framework}.

\begin{proposition} \label{prop:affineincidence}
The  incidence variety equals
$ V_G = \{(s,t) \in \RR^{2\ell} \times \RR^{2\ell} : \delta_G(s) = \delta_G(t)\}$.
\end{proposition}

A cornerstone in rigidity theory is the
{\em Geiringer-Laman Theorem}, which we state next.
For a textbook introduction see \cite[Chapter~8]{SS}.
The {\em rigidity matroid} $\mathcal{R}_\ell$ is the matroid on the ground set $\binom{[\ell]}{2}$
which is given by the
Jacobian matrix $J_\ell$ of the map $\delta_{K_\ell}$.
The matrix $J_\ell$ has format
$2\ell \times \binom{\ell}{2}$, and its rank is
 $2 \ell - 3$. Hence the rigidity matroid
$\mathcal{R}_\ell$ has rank $2\ell-3$.

\begin{figure}[h]
	\centering
	\includegraphics[scale = .35]{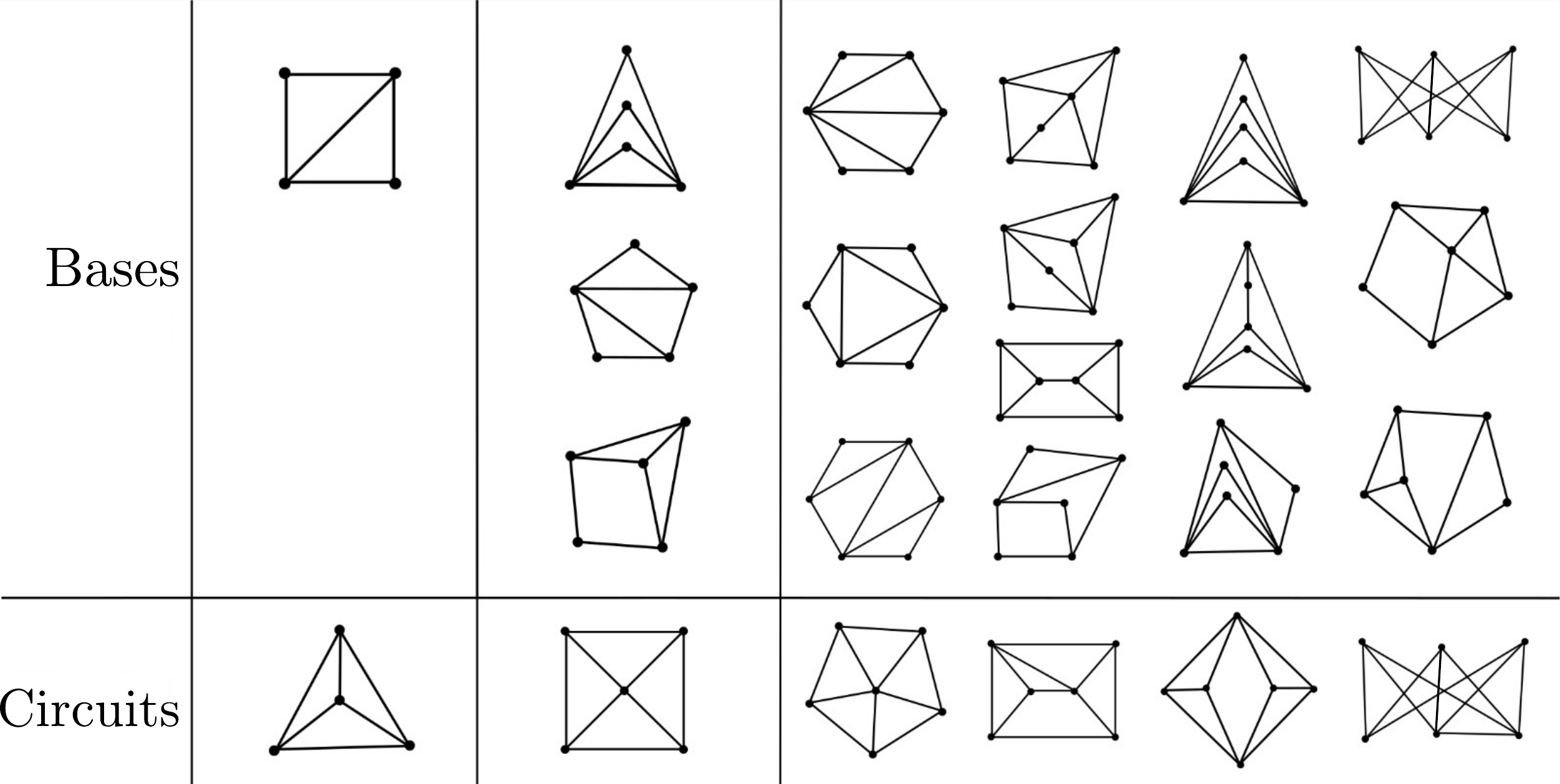}
	\caption{Bases and spanning circuits of the rigidity matroid for $\ell = 4, 5, 6$.  }
	\label{figure:rigidity}
\end{figure}

\begin{theorem}[Geiringer-Laman]\label{thm:Geiringer_Laman}
A graph $G \subseteq \binom{[\ell]}{2}$ is independent in the rigidity matroid $\mathcal{R}_\ell$
if and only if the induced subgraph of $G$ on any set $I\subseteq [\ell]$
 has at most $2|I|-3$ edges.
 The bases of $\mathcal{R}_\ell$ are the maximally independent sets.
 These rigidity bases have $2 \ell-3$ vertices and they are
 obtained from $K_2$ by successively performing the following two operations: \vspace{-0.14cm}
 \begin{itemize}
 \item  add a new vertex and connect it to two old vertices; \vspace{-0.2cm}
 \item  subdivide an edge by a new vertex and connect that to an old vertex.
 \end{itemize}
 \end{theorem}

We write ${\rm rank}(G)$ for the rank of a graph $G$, viewed as a set in the rigidity matroid
$\mathcal{R}_\ell$. A graph is a {\em rigidity circuit} if ${\rm rank}(G) < |G|$ but ${\rm rank}(G') = |G'|$ for every subgraph $G' \subset G$.
Figure \ref{figure:rigidity} shows all unlabeled graphs
corresponding to rigidity bases and  circuits for $\ell \leq 6$.
In rigidity theory, a graph $G \subseteq\binom{[\ell]}{2}$ is 
called \emph{$(a,b)$-sparse}, if 
$\,|\,G[S]\,| \,\leq \,a  |S| - b\,$ holds for every non-empty 
subset $S \subseteq[\ell]$.
Here $a,b \in \NN$ and $G[S]$ is the induced subgraph of $G$ on $S$.

\begin{remark}\label{rem:sparsity}
By Theorem~\ref{thm:Geiringer_Laman}, a graph $G$ is independent
in $\mathcal{R}_\ell$  if and only if it is $(2,3)$-sparse. In this case, $G$ must contain a vertex of degree at most three. Moreover, $G$ is $(2,4)$-sparse if and only if it does not contain any induced bases on at least three vertices. 
\end{remark}

We now show how rigidity theory gives information about the codimension of $V_G$.
To this end, we associate two more varieties to a given graph  $G \subseteq \binom{[\ell]}{2}$.
Following \cite[Section~2]{Raz}, we define $X_G$ as the Zariski closure in ${\rm Gr}(2,4)^\ell$ of all configurations of $\ell$ \textit{distinct} lines satisfying the incidences of $G$. 
The \textit{strict realization} $Y_G$ is the Zariski closure  
of all configurations of $\ell$ \textit{distinct} lines satisfying \textit{exactly} the incidences of $G$, 
and no more. Note that  $Y_G \subseteq X_G \subseteq V_G$, as well as $Y_G \subseteq W_G \subseteq V_G$. 
The rigidity matroid constrains the codimension of $Y_G$ in ${\rm Gr}(2,4)^\ell$.

\begin{proposition}\label{prop:codim_WG}
For any graph $\,G \subseteq\binom{[\ell]}{2}$, we have
        $\,{\rm codim}(Y_G) \,\geq \, {\rm rank}(G)$.
\end{proposition}

\begin{proof}
    We prove this by induction on $\ell$. The case $\ell = 1$  is clear. 
    Suppose $G$ contains a basis of $\mathcal{R}_\ell$ as a subgraph.
    By~\cite[Corollary 5]{Raz}, we have ${\rm codim}(X_G) = 2 \ell - 3 = {\rm rank}(G)$.
     This implies the assertion.
Next  assume that $G$ is basis-free and  $Y_G \neq \emptyset$. 
By Remark~\ref{rem:sparsity}, $G$ has a vertex $v$ of degree $r$ with $0 \leq r \leq 3$. 
Let  $G^{'}$ be the graph obtained from $G$ by deleting~$v$.
The projection $\pi:{\rm Gr}(2,4)^\ell \rightarrow {\rm Gr}(2,4)^{\ell-1}$ forgets the line indexed by $v$. Let $Y $ be any irreducible component of $Y_G$. Then $\pi(Y) \subseteq Y_{G^{'}}$ and hence $Y_{G^{'}} \neq \emptyset$. Since $Y$ is irreducible,
we have
    \begin{equation} \!
        \dim(Y) \leq \dim(\pi(Y)) + \dim(\pi^{-1}(A) \cap Y) \leq \dim(Y_{G^{'}}) + \dim(\pi^{-1}(A) ) \quad
        \hbox{for all $ A \in \pi(Y)$.}
    \end{equation}
     By the induction hypothesis, ${\rm codim}(Y_{G^{'}}) \geq {\rm rank}(G^{'})$. 
     The fiber $\pi^{-1}(A)$ is the intersection of $r$ Schubert divisors in ${\rm Gr}(2,4)$. Since $A \in \pi(Y) \subseteq Y_{G^{'}}$, we can take the lines defining the $r$ divisors to be generic and satisfying precisely the intersection conditions of $G^{'}$. 
     The codimension of $\pi^{-1}(A)$ equals $r$, unless $r=3$ and all lines 
     defining the divisors intersect, by~\cite[Lemma 7]{Raz}. 
    This means that $G$ contains an induced $K_4$, which is a dependent set in the rigidity matroid. Hence, in this case, adding $r=3$ edges to the graph $G^{'}$ increases the rank by at most two.
    This matches the dimension of the fiber for this case. This implies that ${\rm codim}(Y) \geq {\rm rank}(G)$. Since $Y$ was any irreducible component of $Y_G$, the claim follows. 
\end{proof}

We now show that the rigidity matroid precisely encodes the codimension of $X_G$.

\begin{theorem}\label{thm:codim_XG}
For any graph $\,G \subseteq\binom{[\ell]}{2}$, we have
    $\,{\rm codim}(X_G) \,=\, {\rm rank}(G)$.
\end{theorem}

\begin{proof}
We write
     $X_G \subseteq V_G \subseteq\mathbb{R}^{2\ell} \times \mathbb{R}^{2\ell}$ 
     as in Proposition~\ref{prop:affineincidence}.
    Let $\pi  \colon  X_G \rightarrow \mathbb{R}^{2\ell}$ denote the first coordinate projection,
     restricted to $X_G$. Since the diagonal in $\mathbb{R}^{2\ell} \times \mathbb{R}^{2\ell}$ 
     is contained in $V_{[\ell]} \subseteq X_G$, the map $\pi$ is surjective.
      It follows that the dimension of $X_G$ is bounded  below by $2 \ell$ plus the dimension of the generic fiber of $\pi$. The generic fiber of $\pi$ is the same as the generic fiber of $\delta_G$, which has dimension $2\ell - {\rm rank}(G)$, as explained below \eqref{eq:deltamap}. We conclude that ${\rm codim}(X_G) \leq {\rm rank}(G)$.
       The reverse inequality follows from Proposition~\ref{prop:codim_WG} and
    \begin{equation}
        X_G \quad = \bigcup_{E \subseteq \binom{[\ell]}{2} \setminus G} \!\!\! Y_{G \cup E} .
    \end{equation}
    Here the union is over all possible subsets $E$ of non-edges of $G$, including the empty set.
\end{proof}

\begin{corollary}\label{cor:VG_CI_then_G_ind}
    Let $G \subseteq\binom{[\ell]}{2}$ be a graph. Then, ${\rm codim}(V_G) \leq {\rm rank}(G) \leq \min\{|G|, 2 \ell -3\}$.
    Therefore, if $V_G$ is a complete intersection, then $G$ is independent in the rigidity matroid $\mathcal{R}_\ell$.
\end{corollary}

The necessary condition for complete intersection in Corollary~\ref{cor:VG_CI_then_G_ind}
is not sufficient; see Example \ref{ex:K25}. We characterize when $V_G$ is a complete intersection or irreducible in Section~\ref{sec7}.

\begin{example}[A bipartite graph] \label{ex:K25}
Fix $\ell = 7$ and  $G = K_{2,5}$. Then, $G$ is independent in the rigidity~matroid, and the strict realization $Y_G$ has the expected codimension ${\rm rank}(G)  = 10$. However, $V_G$ is not a complete intersection. In fact, ${\rm codim}(V_G) = 9$, because $V_G$ has an
extraneous component obtained by taking a double line and any five lines incident to it.
\end{example}

Rigidity circuits are interesting examples of 
graphs $G$ that are not complete intersections.
One infinite class of rigidity circuits are the wheel graphs 
$W_\ell$. The wheel  $W_\ell$ has $2\ell-2$ edges.
One vertex is connected to each vertex in a cycle of length $\ell-1$.
For instance, $W_4$ equals~$K_4$.

\begin{theorem} \label{thm:wheel}
The wheel graph $G=W_\ell$ has precisely
$2^{\ell-1} - 2(\ell-1)$ irreducible components.
\end{theorem}

\begin{proof}
We color the $\ell-1$ triangles with vertices $\{i,i+1,\ell\}$ black or white and merge adjacent triangles of the same color.
Let $\sigma_1,\sigma_3,\sigma_5,\ldots \subseteq [\ell]$ be the vertices of the black subpolygons
and $\sigma_2,\sigma_4,\sigma_6,\ldots \subseteq [\ell]$ those colored white.
The following subvariety of $V_G$ is irreducible:
\begin{equation}
\label{eq:VVVVVV} V_{\sigma_1  }\,\cap \, V_{\sigma_2 }^* \,\cap \, V_{\sigma_3 } \,
\cap \, V_{\sigma_4 }^* \,\cap \,V_{\sigma_5 } \,\cap \, V_{\sigma_6 }^* \,\cap
 \cdots. 
\end{equation} 
 Here the $\ell$ lines are arranged in a cycle of concurrent and coplanar clusters.
 One extreme case is when the $(\ell-1)$-gon is fully black resp.~fully white,
 in which case (\ref{eq:VVVVVV})  equals $V_{[\ell]}$ resp.~$V_{[\ell]^*}$.
 The other extreme case arises when $\ell$ is odd and the colors alternate:
 $$
 V_{12\ell} \, \cap \,   V_{23\ell}^* \, \cap \,
  V_{34\ell} \, \cap \,   V_{45\ell}^* \, \cap \,
  V_{56\ell} \, \cap \, \cdots\,
\quad  {\rm or}  \qquad
 V_{12\ell}^* \, \cap \,   V_{23\ell} \, \cap \,
  V_{34\ell}^* \, \cap \,   V_{45\ell} \, \cap \,
  V_{56\ell}^* \, \cap \, \cdots
$$
These are the two irreducible components of $W_{G}$.
Note that $W_G = \emptyset$ when $\ell \geq 6$ is even.
For instance, for $\ell= 6$, all $22 = 2^5 - 2\cdot 5$ components
require additional line incidences. The components
$V_{[6]}$ and $V_{[6]}^*$ have codimension $9$;
the other $20$ components have codimension $10$.

Among the $2^{\ell-1}$ colorings of the triangles,
there are $2(\ell-1)$ special colorings, where
 one triangle is colored differently from all the others.
If this happens then its variety
becomes $V_{[\ell]}$ or $V_{[\ell]}^*$. The lonely triangle
has no choice but to switch color.
For all $2^{\ell-1}-2(\ell-1)$ other
colorings, the varieties (\ref{eq:VVVVVV})
are distinct and irreducible. 
Their union equals $V_G$.
\end{proof}

Every triangulated $\ell$-gon $G$ is a rigidity basis,
with $|G| = 2 \ell-3$. Its variety $V_G$ is~a complete intersection
by Theorem \ref{thm:CI}. 
Such triangulations for  $\ell=4,5,6$ appeared in
Examples \ref{ex:whiteandblack}, \ref{ex:5a} and \ref{ex:trianghex}.
We close this section by describing the irreducible decomposition of $I_G$.

\begin{theorem} \label{thm:triangulations}
Let $G$ be the edge graph of any triangulation
of the $\ell$-gon. Then $V_G$ has~$2^{\ell-2}$
irreducible components, each of codimension $2 \ell-3$,
one for each bicoloring of the triangles.
\end{theorem}

\begin{proof}[Sketch of proof]
This is similar to the proof of Theorem~\ref{thm:wheel}.
Black triangles indicate concurrent lines,
while white triangles indicate coplanar lines.
Monochromatic regions are fused to form larger cells
in a subdivision of the $\ell$-gon.
Each irreducible component is then an intersection of
varieties $V_\sigma$ and $V_\tau^*$,
where $\sigma, \tau \subseteq [\ell]$ are cells in 
a subdivision of the $\ell$-gon.
\end{proof}

\section{Irreducibility and Complete Intersection}
\label{sec7}

In this section we determine when $V_G$ is irreducible or a complete intersection. 
Let $\mathcal{P}_{\ell}$ denote the set of all set partitions of $[\ell]$. When
listing elements $J=\{J_1,\dots,J_r \}$ of $\mathcal{P}_{\ell}$, we omit the singletons, 
i.e.~$J_1, \ldots, J_r \subseteq[\ell]$ are disjoint and $|J_i|>1$.
Thus the finest partition is $J = \emptyset$.
Let $|J|:=|J_1|+\cdots+|J_r|$ and  $s_J:=|J|-r$. We write $G_J$ for the graph on $\ell - s_J$ vertices obtained from  $G \subseteq\binom{[\ell]}{2}$ by identifying all vertices within each $J_i$, merging parallel edges, and deleting
resulting loops. If $J = \emptyset$, then $G_J = G$ and $s_J = 0$. An important observation~is 
\begin{equation}\label{eq:union_J}
    V_G \,\,\,\, = \,\,\bigcup_{J \in \mathcal{P}_{\ell}} \tilde{X}_{G_J}  ,
\end{equation}
where $\tilde{X}_{G_J} \subseteq {\rm Gr}(2,4)^{\ell}$ is the variety of $\ell$-tuples in $V_G$ where the lines in each $J_i \in J$ coincide. We can identify $\tilde{X}_{G_J}$ with $ X_{G_J} \subseteq{\rm Gr(2,4)}^{\ell - s_J}$.
Note that ${\rm codim}(\tilde{X}_{G_J})={\rm codim}(X_{G_J})+4s_J$. 

\begin{theorem}[Codimension]\label{thm:codim}
For any graph $G \subseteq\binom{[\ell]}{2}$, the incidence variety $V_G$ satisfies
    \begin{equation}\label{eq:cod}
      \,  {\rm codim}(V_G) \,\,\,=\,\,\, \min_{J \in \mathcal{P}_{\ell}} \big\{ {\rm rank}(G_{J})+4  s_J   \big\} \, .
    \end{equation}
If $G$ is $(2,3)$-sparse then    
     \begin{equation}\label{eq:cod_2}
        {\rm codim}(V_G) \,\,\,=\,\,\, \min_{J \in \mathcal{P}_{\ell}} \big\{ |G_{J}|+4  s_J   \big\} \, . \qquad
    \end{equation}
\end{theorem}

\begin{proof}
Equation (\ref{eq:cod}) follows by combining
\eqref{eq:union_J} and Theorem~\ref{thm:codim_XG}.
Suppose now that $G$ is $(2,3)$-sparse. We claim that if
$J \in \mathcal{P}_{\ell}$ satisfies ${\rm rank}(G_{J})+4  s_J \leq {\rm rank}(G)$ then $G_J$ is also $(2,3)$-sparse.
For the proof, we first  note that
$G_J$ is obtained from $G$ by iteratively identifying pairs of vertices. Since induced subgraphs of a $(2,3)$-sparse graph are
 (2,3)-sparse, it suffices to show:
  if $G'$ is obtained from $G$ by identifying two vertices then  $G'$ is $(2,3)$-sparse.

Let $\delta:= |G|-|G^{'}|$. Since $G$ is $(2,3)$-sparse, we have
    \begin{equation}\label{eq:2_a_sparse}
        |G^{'}| \,=\, |G| - \delta \,\leq \,2 \ell -3-\delta \,\,=\,\, 2(\ell -1)-3+(2-\delta)  \, .
    \end{equation}
    Next, we record the inequalities
$\,    
    {\rm rank}(G) - \delta\, \leq \,{\rm rank}(G')\, \leq\, {\rm rank}(G) - 4$.
        The first inequality is due to the submodularity of the rank.
        The second holds by our assumption that ${\rm rank}(G_{J})+4  s_J \leq {\rm rank}(G)$.
        Therefore,     $\delta \geq 4$.
          Together with (\ref{eq:2_a_sparse}), this implies that $G'$ is (2,3)-sparse. 

We now establish the identity (\ref{eq:cod_2}).
   Let $J \in \mathcal{P}_{\ell}$ and consider ${\rm rank}(G_{J})+4  s_J$ as in~\eqref{eq:cod}. By 
   the claim we just proved, either ${\rm rank}(G_{J})+4 s_J > {\rm rank}(G) = |G|$, in which case $G_J$ does not contribute to the minimum in~\eqref{eq:cod}, or $G_J$ is $(2,3)$-sparse and hence ${\rm rank}(G_J) = |G_J|$. 
\end{proof}

We now introduce a class of graphs that is important for the two main theorems in this section.
 A graph $G \subseteq\binom{[\ell]}{2}$ is  \textit{contraction-stable} (CS) if, for every partition $J \in \mathcal{P}_{\ell}$, we have
    \begin{equation}\label{eq:balanced}
       |G_{J}|+4  s_J \,\,\geq \,\,|G|  \, .
    \end{equation}
We call $G$  \textit{strictly contraction-stable} (SCS) if the inequality \eqref{eq:balanced} is  strict 
 whenever $J \not = \emptyset$.

\begin{example}
    The bipartite graph $G=K_{2,4}$ is $(2,4)$-sparse and hence independent in $\mathcal{R}_6$.
    But, $G$ is not strictly contraction-stable. Indeed, if $J=\{ \{1,2\}\}$, where $1,2$ are the four-valent vertices,
    then $s_J = 1$ and~\eqref{eq:balanced} becomes an equality. This corresponds to a second component $\tilde{X}_{G_J}$ of $V_{G}$, which is different from $Y_G$. Points in $\tilde{X}_{G_J}$ consist of $4$ lines incident to a double line. 
    This is similar to
    Example~\ref{ex:K25}, however $V_G$ is still a complete intersection. 
\end{example}

Let $K^{'}_{2,r}$ denote the graph obtained from $K_{2,r}$ by adding an edge between the two $r$-valent vertices.
A similar analysis for $K_{2,r}$ and $K_{2,r}^{'}$ for small $r$ yields the following corollary. 

\begin{corollary}\label{cor:2_cont_rig}
    If $G$ is strictly contraction-stable, then $G$ is $K_{2,4}$-free and $K_{2,3}^{'}$-free. If $G$ is contraction-stable, then $G$ is $K_{2,5}$-free and $K_{2,4}^{'}$-free.
\end{corollary}

Our first main theorem in this section is the characterization of complete intersections.

\begin{theorem}[Complete Intersection]\label{thm:CI}
The incidence variety 
    $V_G$ is a complete intersection if and only if the graph $G$ is $(2,3)$-sparse and contraction-stable.
\end{theorem}

\begin{proof} 
    By Theorem~\ref{thm:codim}, if the graph $G$ is both $(2,3)$-sparse and contraction-stable then ${\rm codim}(V_G) = {\rm rank}(G) = |G| $. By Corollary~\ref{cor:VG_CI_then_G_ind}, if $G$ is not $(2,3)$-sparse then $V_G$ is not a complete intersection. Assume that $G$ is $(2,3)$-sparse but not contraction-stable. By~Theorem~\ref{thm:codim}, ${\rm codim}(V_G) < {\rm rank}(G) = |G|$
    and so $V_G$ cannot be a complete intersection.
\end{proof}

\begin{corollary}\label{cor:add_edge_drop_dim}
    Let $G$ be $(2,4)$-sparse and strictly contraction-stable.
     Let $e \in \binom{[\ell]}{2} \backslash G$ be a non-edge of $G$ and 
     denote by $\tilde{G} := G \cup \{e\}$. Then, ${\rm codim}(V_{\tilde{G}}) = {\rm codim}(V_G) +  1$.
\end{corollary}

\begin{proof}
    Note that $\tilde{G}$ is contraction-stable. Since $G$ is $(2,4)$-sparse, $\tilde{G}$ is $(2,3)$-sparse 
    and hence independent. Then, by Theorem~\ref{thm:CI}, 
    ${\rm codim}(V_{\tilde{G}}) = |\tilde{G}| = |G| + 1 = {\rm codim}(V_{G}) + 1$.
\end{proof}

We now move towards understanding irreducibility. We start with a necessary condition.

\begin{proposition}\label{prop:bases_red}
    If $G$ is not $(2,4)$-sparse, then $V_G$ is reducible. 
\end{proposition}

\begin{proof}
Suppose $G$ is not $(2,4)$-sparse. Then $G$ has a subgraph $H $ which is a basis
of $\mathcal{R}_k$ for some $k \geq 3$. The projection $\pi: V_G \to V_H$ takes the $\ell $ lines to the
 lines in $[k]$. We will show that the Zariski closure $V$
of the image $\pi(V_G)$ is reducible.
This implies that $V_G$ is reducible.

By Corollary~\ref{cor:VG_CI_then_G_ind} and the fact that $H$ is a basis, we have
${\rm codim}(V_H) = |H|=2k-3$. Hence, both $V_{[k]}$ and $V_{[k]}^\ast$ are irreducible components of $V_H$. We claim that $V_{[k]}$ and $V_{[k]}^*$ are also irreducible components of $V$. Since $V \subseteq V_H$, it suffices to show that $V_{[k]}$ and $V_{[k]}^*$ are subsets of $V$. Let $A' = (A_1', \ldots, A_k') \in V_{[k]}$. We lift the configuration $A'$ to a point $A$ in ${\rm Gr}(2,4)^\ell$ as follows. The $i$-th line $A_i$ is taken to be $A_i'$ if $i \in [k]$, else it is just $A_1'$. 
This definition ensures   $A \in V_{K_\ell} \subseteq V_G$, and  we have $\pi(A) = A'$ by construction.
This shows that $V_{[k]} \subseteq V$. An analogous argument proves the other inclusion $V_{[k]}^\ast \subseteq V$.    
\end{proof}

\begin{proposition}\label{prop:YG_24_sparse}
    Let $G \subseteq\binom{[\ell]}{2}$ be a $(2,4)$-sparse graph. Then, $Y_G$ is unirational and
    \begin{equation}
        {\rm codim}(Y_G) \,\,=\,\, {\rm rank}(G) \,\,=\,\, |G| \, .
    \end{equation}
\end{proposition}

\begin{proof}
    Like Proposition~\ref{prop:codim_WG}, we prove this by induction on $\ell \geq 1$. The case $\ell = 1$ is clear. 
    Let $\ell \geq 2$. By Remark~\ref{rem:sparsity}, $G$ has a vertex $v$ of degree $r \in \{0,1,2,3\}$.
    Let     $G^{'}$ be the graph obtained from $G$ by deleting $v$. 
        Since $G^{'}$ is $(2,4)$-sparse, by the induction hypothesis, $Y_{G^{'}} $ is unirational 
        and ${\rm codim}(Y_{G^{'}}) = {\rm rank}(G^{'}) = |G^{'}|$. We consider the projection $\pi: V_G \rightarrow V_{G^{'}}$.    

    Since $G$ is $(2,4)$-sparse, it is triangle-free.
    Generic points in $Y_{G^{'}}$ are distinct lines satisfying the strict incidences of $G$. Hence, the fiber of $\pi$ over a generic point in $Y_{G^{'}}$ is an intersection of $r $ Schubert divisors in ${\rm Gr}(2,4)$ associated to non-intersecting lines. Since $r \leq 3$,     the generic fiber of $\pi$ is a variety of codimension $r$  in ${\rm Gr}(2,4)$ which is unirational, hence irreducible.
         Therefore, there exists (by composing parametrizations) a unirational component $Y $ of $V_G$,
         having codimension equal to $|G^{'}|+r = |G|={\rm rank}(G)$, which maps onto $Y_{G^{'}}$ via $\pi$.
     
      We claim that $Y=Y_G$. First of all,
      generic points in $Y$ correspond to distinct lines, i.e.~$Y \subseteq X_G$. 
      We now argue that generic points on $Y$ correspond 
      to strict realizations of $G$, i.e.~$Y \subseteq Y_G$. If not, then
      there exists a graph $\tilde{G} \subseteq \binom{[\ell]}{2} $, obtained 
      from $G$ by adding at least one non-edge of $G$, 
      such that $Y \subseteq X_{\tilde{G}}$. But then by Theorem~\ref{thm:codim_XG},
       we would have ${\rm codim}(Y) \geq {\rm codim}(X_{\tilde{G} }) = {\rm rank}(\tilde{G}) > {\rm rank}(G) $, 
       which contradicts the previous paragraph.      
        Lastly, the Zariski open set of points in $Y_G$ which have exactly the incidences prescribed by $G$ are in the image of the map which parametrizes $Y$. We conclude that $Y = Y_G$. 
\end{proof}

We next present our second main result in this section.

\begin{theorem}[Irreducibility]\label{thm:irred}
The incidence variety
    $V_G$ is irreducible if and only if $G$ is $(2,4)$-sparse and strictly contraction-stable. 
    In this case, we have $V_G =W_G=  X_G   = Y_G $,
    and this irreducible variety is a unirational complete intersection.
\end{theorem}

\begin{proof}
     If the graph $G$ is not $(2,4)$-sparse, then $V_G$ is reducible by Proposition~\ref{prop:bases_red}.
        Assume that $G$ is $(2,4)$-sparse but not SCS. By Proposition~\ref{prop:YG_24_sparse}, the variety
        $Y_G$ is unirational and ${\rm codim}(Y_G) = |G|$. Since $G$ is not SCS, 
        there exists a partition $J \in \mathcal{P}_{\ell}$ and an irreducible component $\tilde{V} \subseteq \tilde{X}_{G_J}$, such that ${\rm codim}(\tilde{V}) = {\rm rank}(G_J) = |G_J| \leq |G|-4s_J = {\rm codim}(Y_G) -4s_J$. Generic points in $\tilde{V}$ have lines which coincide precisely according to~$J$, and 
        we have ${\rm codim}(\tilde{V}) \leq {\rm codim}(Y_G)$. Thus $\tilde{V} \not \subset Y_G$ and so $V_G$ must be reducible.

    Conversely, if $G$ is $(2,4)$-sparse and SCS, then $V_G$ is a complete intersection
    by Theorem~\ref{thm:CI}. Let $V \subseteq V_G$ be any irreducible component. If generic points in $V$ are distinct lines   with precisely the incidences of $G$, then $V \cap Y_G$ is Zariski open in $Y_G$, so $V=Y_G$. If not, then $V \subseteq V_{\tilde{G}}$, where $\tilde{G}$ is obtained from $G$ by adding at least one non-edge. By Corollary~\ref{cor:add_edge_drop_dim}, $V$ has positive codimension in $V_G$. This is a contradiction to the complete intersection $V_G$ being
         equidimensional.
          We conclude that $V=V_G = Y_G$ is a unirational complete intersection.
\end{proof}

The previous result allows us to give an easy-to-use sufficient condition for irreducibility.

\begin{corollary}
Let $G$ be a $(2,4)$-sparse graph with at most two independent cycles. Then $G$ is strictly contraction-stable,
and hence the incidence variety $V_G$ is irreducible.
\end{corollary}

\begin{proof}
The cycles hypothesis implies
     $|G| \leq \ell+ 1 $. Suppose there exists $J \in \mathcal{P}_{\ell}$ which 
     violates the strict inequality in~\eqref{eq:balanced}. Since $G_J$ is connected, 
     we have $|G_J| \geq \ell - s_J -1$. Combining the two inequalities, we obtain $3s_J \leq 2$, which implies $s_J = 0$. Therefore, $G$ is SCS.
     \end{proof}

\begin{remark}
    An analogous argument proves the following extension: If the graph
        $G$ is $(2,4)$-sparse, $K_{2,4}$-free, and has at most 
        five independent cycles, then $V_G$ is irreducible.
\end{remark}

The SCS property for small graphs
can be captured by forbidding certain subgraphs. 
The following proposition can easily be proved by exhaustive search. 

\begin{proposition}\label{prop:until_l_9}
Fix $\ell \leq 9$ and let
     $G \subseteq \binom{[\ell]}{2}$ be a $(2,4)$-sparse graph.
    Then $G$ is strictly contraction-stable, and hence $V_G$ irreducible, if and only if $G$ is $K_{2,4}$-free.
\end{proposition}

\begin{example}\label{eg:l_10}
Proposition~\ref{prop:until_l_9} does not hold for $\ell \geq 10$.
The graph $G$ in Figure~\ref{figure:eliaexample} shows that the if-direction fails.
The variety $V_G$ is a complete intersection of degree $2^{15} = 32768$. It has two irreducible components.
The main component $Y_G$ has degree $32640$.
An extraneous component isomorphic to $\tilde{X}_{G_J}$  has degree $128$.
Here the partition is $J=\{\{1,2,3\}\}$.
Note that $|G_J|=7 = 15-8 = |G|-4s_J$. 
For points in $\tilde{X}_{G_J}$, the lines $1,2,3$  coincide.
\end{example}

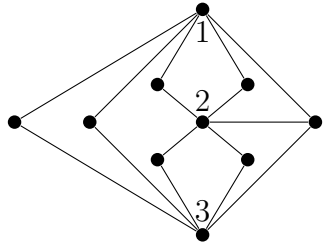
\begin{figure}[h]
\centering
\begin{tikzpicture}[scale=1, every node/.style={circle, fill=black, inner sep=1.8pt}]

    \tikzset{dot/.style={circle, fill=black, inner sep=1.8pt}}

    \node (top) at (0,3) {};    
    \node (midleft1)  at (-0.6,2) {};
    \node (midright1) at (0.6,2)  {};
    \node (center)    at (0,1.5) {}; 
    \node (midleft2)  at (-0.6,1) {};
    \node (midright2) at (0.6,1)  {};
    \node (bottom)    at (0,0) {};  
    
    \node (left1)  at (-1.5,1.5) {};
    \node (left2)  at (-2.5,1.5) {};
    \node (right) at ( 1.5,1.5) {};
    
    \draw (top) -- (midleft1);
    \draw (top) -- (midright1);
    \draw (midleft1) -- (center);
    \draw (midright1) -- (center);
    
    \draw (center) -- (midleft2);
    \draw (center) -- (midright2);
    \draw (midleft2) -- (bottom);
    \draw (midright2) -- (bottom);
    
    \draw (top) -- (left1) -- (bottom) -- (right) -- (top);
    \draw (top) -- (left2) -- (bottom) ;
    
    \draw (center) -- (right);

    \draw (0,2.7) node[draw=none, fill=none, inner sep=0] {$1$};

    \draw (0,1.78) node[draw=none, fill=none, inner sep=0] {$2$};

    \draw (0,0.32) node[draw=none, fill=none, inner sep=0] {$3$};

\end{tikzpicture}
\caption{This graph with $\ell \!=\! 10$ 
 is $(2,4)$-sparse, $K_{2,4}$-free, and not strictly contraction-stable. }
    \label{figure:eliaexample}
\end{figure}

\begin{theorem} \label{thm:isprime}
    If the incidence variety $\,V_G$ is irreducible, then the ideal $I_G$ is prime.
\end{theorem}

\begin{proof}
    By Theorem~\ref{thm:irred}, $V_G$ is a complete intersection.
    Hence all associated primes of $I_G$ have the same dimension.
    Since $V_G$ is irreducible, 
      the radical ideal  $\sqrt{I_G}$ is  prime. In particular, $I_G$ is primary, of some length $\lambda$, 
      over the prime ideal $\sqrt{I_G}$.  The multidegree of $I_G$ equals 
      $\lambda$ times the multidegree of $\sqrt{I_G}$. 
      To establish that $I_G$ is prime, it suffices to prove $\lambda = 1$.

    We show this by induction on $\ell$, using Corollary~\ref{cor:CI}. The 
    case $\ell = 1$ is clear. Let $\ell \geq 1$. As in the proof of Theorem~\ref{thm:irred}, 
    fix a vertex $v_\ell$ of degree $r \leq 3$ and denote by $G^{'}$ the graph obtained from $G$
     by deleting $v_\ell$. 
     Let $v_1,\ldots,v_r$ be the vertices in $G$ adjacent to $v_\ell$. Pick a term $\gamma_u \, T_u$ with $T_u=\prod_{i=1}^\ell t_i^{5-u_i}$ in the multidegree of $I_{G}$ which involves $u_\ell = 4$ Schubert conditions on $v_\ell$. Such term exists by checking~\eqref{eq:mutlidegCI}. By standard Schubert calculus, we have that $\lambda \gamma_u = 2  \lambda^{'}  \gamma^{'}_u$, where $\gamma^{'}_u$ is the coefficient
      of the monomial $t_\ell^{-1} \, \prod_{i=1}^r t_i \, T_u$
      in the multidegree of $\sqrt{I_{G^{'}}}$, and $\lambda^{'}$ is such that the multidegree of $I_{G^{'}}$ 
      equals $\lambda^{'}$ times the multidegree of $\sqrt{I_{G^{'}}}$. By the induction hypothesis, $I_{G^{'}}$ is prime, and hence $\lambda^{'}=1$. On the other hand, from~\eqref{eq:mutlidegCI} we can read off $\gamma_u = 2  \gamma_u^{'}$. We conclude that $\lambda=1$, and hence that $I_G$ is prime.
\end{proof}

An interesting problem is determining which graph $G$ has a realization ($W_G \not= \emptyset$) or a strict realization ($Y_G \neq \emptyset$). The smallest graph that is not strictly realizable is $G=K_{2,3}^{'}$. This graph appeared in Example~\ref{ex:5a}, where we found that $Y_G = \emptyset$, but $W_G \neq \emptyset$. The smallest triangle-free graph $H$ that is not realizable is $K_{4,4}$ with one edge removed, see Example \ref{eg:K44_1edge}, for which $Y_H=W_H=\emptyset$. In general, these problems are related to \emph{incidence theorems}~\cite{FP}.

\begin{example}
    Other graphs $G$ with $Y_G = \emptyset$ are the wheel $W_\ell$ with even $\ell \geq 6$ (see Theorem \ref{thm:wheel}) and the `winner' graph for $\ell=7$ in Proposition~\ref{prop:winner7}. We can construct
         many other graphs which are not strictly realizable. For example,         
         let $G$ be obtained by gluing $r \geq 3$ complete graphs $K_{\ell_i}$ with $\ell_i \geq 3$ and $i=1,\dots,r$ along a common edge. Then, $Y_G = \emptyset$, but $W_G \neq \emptyset$.
\end{example}

We conclude by conjecturing a sufficient condition for strict realizability.

\begin{conjecture}
    Let $G$ be a $K_{2,3}^{'}$-free graph which is
        independent in $\mathcal{R}_\ell$. Then $Y_G \neq \emptyset$.
\end{conjecture}

\section{Spanning Tree Coordinates}
\label{sec4}

In this section we introduce the coordinates
we find most useful for symbolic computations with 
incidence varieties of lines in $3$-space.
Let $T$ be a tree on $[\ell]$.
This is a connected graph with $\ell-1$ edges $e \in \binom{[\ell]}{2}$.
For each edge $e$ of $T$ we introduce a $2 \times 2$ matrix 
$X_e$ whose $4$ entries are unknowns. 
Let $\CC[X]$ denote the polynomial ring in these $4\ell-4$ variables.
The $2 \times 2$ determinants $ {\rm det}(X_e)$
are quadrics in this polynomial ring.  Since they use distinct variables,
they generate a prime ideal
which is a complete intersection of codimension $\ell-1$.

We now fix an arbitrary connected graph $G$ with vertex set $[\ell]$.
Choose a spanning tree $T$ of $G$, and let $\CC[X]$ be the
polynomial ring defined in the previous paragraph.
For any edge $g$ of $G$ that is not in $T$,  consider the
graph $T \cup \{g\}$. This graph has a unique cycle,
consisting of the edge $g$ and some edges $e_1,\ldots,e_r$
from the tree $T$. We define the $2 \times 2$ matrix
$$ Y_g \,\,=\,\, \sum_{i=1}^r X_{e_i} \, . $$
With the pair $(G,T)$ we associate the following ideal of quadrics in the polynomial ring $\CC[X]$:
\begin{equation}
\label{eq:GTideal}
I_{G,T} \,\,\, = \,\,\,
\bigl\langle \,{\rm det}(X_e) \,: \, e \in T \, \bigr\rangle \,+\,
\bigl\langle \,{\rm det}(Y_g) \,: \, g \in G \backslash T  \,\bigr\rangle  \, .
\end{equation}

\begin{example}[Cycle]
Let $G$ be the $\ell$-cycle with spanning tree $T = \{12,23,\ldots,(\ell-1)\ell\}$
and extra edge $g = 1\ell$. The polynomial ring $\CC[X]$
has $4\ell-4$ variables, namely the entries of
the $2 \times 2$-matrices $X_1,X_2,\ldots,X_{\ell-1}$ for the chain $T$.
The ideal defined above equals
$$ I_{G,T} \,\,\, = \,\,\,
\bigl\langle \, {\rm det}(X_1),\,
{\rm det}(X_2),\,\ldots,\,{\rm det}(X_{\ell-1}),\,
{\rm det}(X_1+X_2 + \cdots+X_{\ell-1})\, \bigr\rangle \, .$$
This ideal is prime and it is a complete intersection. This can be
shown by arguing that the variety of $I_{G,T}$ is
irreducible of codimension $\ell$ in $\CC^{4\ell-4}$, and the $\ell$  generators of $I_{G,T}$
form a Gr\"obner basis with square-free leading terms for the
reverse lexicographic term order.
\end{example}

\begin{example}[Complete graph] \label{ex:compgra2}
Let $G = K_\ell$ and let $T$ be a star tree on $[\ell]$. Then have
$$ I_{G,T} \,\, = \,\,
\bigl\langle \, {\rm det}(X_i)\,:\, 1\leq i \leq \ell-1 \bigr\rangle
\,+\,
\bigl\langle \, {\rm det}(X_i+X_j)\,:\, 1\leq i < j \leq \ell-1 \bigr\rangle \, .
$$
This ideal is radical for $\ell=3$ and not radical for $\ell\geq 4$.
It has two associated primes, namely
$$ \bigl\langle \hbox{$2 \times 2$-minors of} \ 
\bigl( X_1 \,|\, X_2 \,|\, \cdots\, |\, X_{\ell-1} \,\bigr) \bigr\rangle \quad {\rm and} \quad
\bigl\langle \hbox{$2 \times 2$-minors of} \ \bigl( X_1^T \,|\, X_2^T \,|\, \cdots\, |\, X_{\ell-1}^T \,\bigr)  \bigr\rangle \, .
$$
Each prime ideal is generated by the maximal minors of a
$2 \times (2\ell-2)$ matrix of unknowns. We
 obtain the matrices by concatenating the $\ell-1$ given $2 \times 2$ matrices, or their transposes.
\end{example}

We now show that $I_{G,T}$ is   isomorphic to  the affine ideal $\widetilde I_G$
from Section \ref{sec2}. In particular, $I_{G,T}$ is independent of
the spanning tree $T$, up to a linear change of coordinates.
It depends only on the graph $G$.
We thus obtain an 
effective representation of our incidence variety $V_G$ 
in ${\rm Gr}(2,4)^\ell$. This respects the scheme structure,
by Proposition \ref{prop:WeFind}. The entries
of the  $2 \times 2$ matrices $X_e$,
indexed by the tree edges $e \in T$, are 
the {\em spanning tree coordinates} for $V_G$.

We construct
the desired isomorphism from the affine coordinates in
(\ref{eq:ABCDE}). Relabel the graph so that
$1,2,\ldots,\ell$ is a topological ordering for the tree $T$, i.e.
for each vertex $k \in [\ell]\backslash \{1\}$
there is a unique vertex  $j_k \in [k-1]$ such that
$(j_k,k)$ is an edge in $T$.
For the first vertex $1$ we replace $\,(\alpha_1 ,\alpha_2,\alpha_3,\alpha_4)$ by
$(0,0,0,0)$ in the matrix ${\bf A}$ in (\ref{eq:ABCDE}). Now, the first line has 
the Pl\"ucker vector $e_{12}$. For each $k \geq 2$, with corresponding
$2 \times 4$ matrix ${\bf B}$, we replace the right $2 \times 2$ block $\beta$ 
by $X_{k-1}$ minus the $2 \times 2$ matrix previously defined for vertex~$j_k$.
We finally add four new unknowns $y_1,y_2,y_3,y_4$ to the polynomial ring
$\CC[X_1,\ldots,X_{\ell-1}]$. These variables do not appear in any of
our polynomials and are just used to get up to $4\ell = {\rm dim}({\rm Gr}(2,4))^\ell$.

\begin{proposition}
Fix a graph $G$ and a spanning tree $T$.
The polynomial ring with $4 \ell $ variables in Section \ref{sec2} is isomorphic to
$\CC[X_1,\ldots,X_{\ell-1}] $ under the map given above. This isomorphism
identifies the ideal $\widetilde I_G$ of the
incidence variety $V_G$ with the ideal
$I_{G,T}$ in (\ref{eq:GTideal}).
\end{proposition}

\begin{proof}
Consider the $2 \times 2$ matrices in (\ref{eq:AdotB})
that correspond to the $\ell-1$ edges of the tree $T$. The 
$4 \ell - 4$ entries of these matrices are
independent linear forms. Our substitution 
replaces these linear forms by variables, namely
the entries of the matrices $X_1,\ldots,X_{\ell-1}$.
For every edge $g$ in $G\backslash T$,
the $2 \times 2$ matrix in (\ref{eq:AdotB}) sums
to zero with the matrices of the edges
$e_1,\ldots,e_r$ that create the cycle
with $g$ in $T \cup \{g\}$.
Under our substitution, this matrix becomes $-Y_g$.
\end{proof}

By abuse of notation, the symbol
$I_G$ now denotes the ideal $I_{G,T}$, whenever the tree $T$
is understood. The {\tt degree} of 
that ideal $I_G$, as computed by {\tt Macaulay2},
equals the total degree in Corollary~\ref{cor:subscheme}.
For instance, we now simply write $I_{K_\ell}$
for the ideal $I_{G,T}$ in Example~\ref{ex:compgra2}.
The two determinantal primes
correspond to $I_{[\ell]}$ and $I_{[\ell]}^*$ in Example \ref{ex:complete}.
See also Example~\ref{ex:12affine}. 

In the remainder of this section, we demonstrate 
the effectiveness of the spanning tree coordinates when
performing symbolic computations for the incidence varieties $V_G$.
For $\ell = 6$, when the polynomial ring $\CC[X]$ has $20$ variables,
the computation is fast in  {\tt Macaulay2}.

\begin{example}[Bipartite graph] \label{ex:K24}
Let $\ell = 6$ and fix $G = K_{2,4}$. For our ideal we can take
$$ \begin{matrix} \!\!\!\!\!\!\!\! I_G \,\, = &  \bigl\langle\,
 {\rm det}(X_1), \,{\rm det}(X_2),\, {\rm det}(X_3),\, {\rm det}(X_4),\, {\rm det}(X_5),\,\qquad \\
& {\rm det}(X_1+X_4+X_5),\,
{\rm det}(X_2+X_4+X_5),\,
{\rm det}(X_3+X_4+X_5)\,
\bigr\rangle \, . \end{matrix}
$$
This ideal is a complete intersection and radical.
Namely, $I_G$ equals the intersection of the prime ideal
$P = \langle {\rm det}(X_1), {\rm det}(X_2), {\rm det}(X_3), {\rm det}(X_4)  \rangle
+ \bigl\langle X_4 \,+ X_5 \bigr\rangle$
and the prime ideal $I_G:P$.
In addition to the eight generators of $I_G$, 
the quotient $I_G:P$ has one extra quartic generator.
Its variety $W_G$ is the incidence variety for
the classical Schubert problem in Example \ref{ex:onevertex}.
Points in $W_G$ are quadruples of lines in $\PP^3$ 
together with the pair of lines that meet them.
On the extraneous component given by $P$,
that pair of lines gets replaced by a double line.
\end{example}

\begin{example}[Triangulated hexagon] \label{ex:trianghex}
The ideal $I_G$  is computed in {\tt Macaulay2} as follows:
\begin{verbatim}
R = QQ[a1,a2,a3,a4,a5,b1,b2,b3,b4,b5,c1,c2,c3,c4,c5,d1,d2,d3,d4,d5];
X1 = matrix {{a1,b1},{c1,d1}}; X2 = matrix {{a2,b2},{c2,d2}}; 
X3 = matrix {{a3,b3},{c3,d3}}; X4 = matrix {{a4,b4},{c4,d4}}; 
                               X5 = matrix {{a5,b5},{c5,d5}};
IG = ideal( det(X1), det(X2), det(X3), det(X4), det(X5),
           det(X1+X2), det(X2+X3), det(X3+X4), det(X4+X5) );
betti mingens IG, codim IG, degree IG
DIG = decompose IG; toString DIG
intersect(DIG) == IG
\end{verbatim}
 We find that the radical ideal $I_G$ is the intersection of
$16$ prime ideals of codimension $9 = |G|$.
Here $G$ is one of the two triangulations of the hexagon.
This confirms Theorem \ref{thm:triangulations} for $\ell=6$.
\end{example}

\section{Six, Seven and Eight Lines}
\label{sec6}

We now return to the classification of the
 incidence varieties $V_G$.
In Section \ref{sec2} we computed all of them for $\ell \leq 5$.
We quickly saw that the structure of the ideals $I_G$ can be quite rich and interesting.
The combinatorics of the multidegrees in Section \ref{sec3}  further underscores this. 

This section is dedicated to much harder computations.
We compute the incidence varieties for graphs with up to $\ell = 8$ vertices.
This involves
polynomials in $24$, $28$, and $32$ variables in the affine coordinates from Section \ref{sec2}.
 We seek to compute the dimension, degree, and primary decomposition of $I_G$.  
  For most graphs $G$ with $\ell \geq 7$, this is intractable with symbolic computation, even if we leverage the spanning tree coordinates in Section \ref{sec4}. 
  
  We instead use tools from \emph{numerical algebraic geometry} (cf.~\cite{BT}).
   This allows us to compute a \emph{numerical irreducible
   decomposition} for all connected graphs up to $\ell = 7$ and
   all connected triangle-free graphs up to $\ell = 8$.
    Our findings are summarized in Tables \ref{table:all-connected-graphs} and \ref{table:triangle-free-graphs}.

\begin{table}[h]
	\centering
	\small
	\begin{tabular}{|c|p{2.2cm}|p{2.0cm}|p{1.8cm}|p{1.8cm}|}
		\hline
		$\ell$ & Connected Graphs & Complete Intersection & Irreducible & Realizable ($W_G \not= \emptyset$) \\
		\hline
		\hline
		4 & 6     & 5      & 3  &   6 \\
		\hline
		5 & 21    & 16    & 6  &  21 \\
		\hline
		6 & 112  & 69   & 17 &  103 \\
		\hline
		7 & 853 & 379 & 52 &  681 \\
		\hline
	\end{tabular}
	\caption{Numerical irreducible decomposition of $V_G$ for connected graphs $G$ with $\ell = 4,5,6,7$.
	The four columns report the number of all graphs $G$
	 up to isomorphism, the number whose variety  $V_G$ is a complete intersection,
	 whose $V_G$ is irreducible, and whose
	 $W_G$ is non-empty.}
	\label{table:all-connected-graphs}
\end{table}

\begin{table}[h]
	\centering
	\small
	\begin{tabular}{|c|p{2.2cm}|p{2.0cm}|c|c|}
		\hline
		$\ell$ & $\mbox{Triangle-Free}$ Graphs & Complete Intersection & Irreducible & Realizable \\
		\hline
		\hline
		4 & 3    &   3       & 3     &  3 \\
		\hline
		5 & 6    &   6       & 6     &  6 \\
		\hline
		6 & 19  &   19      & 17   &  19 \\
		\hline
		7 & 59  &   57      & 52  &  59 \\
		\hline
		8 &267 &   254   & 219 &  266 \\
		\hline
	\end{tabular}
	\caption{Numerical irreducible decomposition of $V_G$ for triangle-free connected graphs $G$ with $\ell = 5,6,7,8$.
	The numbers in the columns report the same properties as in Table \ref{table:all-connected-graphs}.}
	\label{table:triangle-free-graphs}
\end{table}

\smallskip
The numerical irreducible decomposition of our incidence variety $V_G$ gives a \emph{witness set} for each component $V$.
 This is a pair $(W, L)$, where $L$ is a generic linear space such that ${\rm codim }(V) + {\rm codim}(L) = 6 \ell$ 
 and $W =  V \cap L$ is a finite set of points such that $|W| = {\rm degree}(V)$. 
 Thus the dimension and degree of each component of $V_G$ are
 read off from the witness set.  We say that a non-edge $ij \notin G$ {\em vanishes numerically} on an
 irreducible component $V$ of $V_G$ if
 \begin{equation}
 \label{eq:tolerance}
\qquad \qquad
{\rm max}_{p \in W} |A_iA_j(p)|  < \varepsilon \, , \quad
\qquad \hbox{for some tolerance $\varepsilon>0$} \, .
\end{equation}
A graph $G$ is listed as {\em realizable} in Tables \ref{table:all-connected-graphs}
or \ref{table:triangle-free-graphs} if (\ref{eq:tolerance})
fails to hold for all non-edges $ij$, for a suitable $\epsilon$.
 We now discuss highlights
 from our data,  beginning  with triangle-free graphs.

\begin{example}[$\ell = 6$, triangle-free]
Among the $19$ triangle-free graphs $G$, there are
$17$ for which $V_G$ is irreducible. 
 The two others are the bipartite graphs $K_{2,4}$ and $K_{3,3}$. They are both $(2,3)$-sparse and CS, hence they give a complete intersection by Theorem \ref{thm:CI}.
$K_{2,4}$ is (2,4)-sparse but not SCS, while $K_{3,3}$ is SCS but not (2,4)-sparse. Hence they are both reducible by Theorem \ref{thm:irred}. The decomposition of
$V_{K_{2,4}}$ was discussed in Example~\ref{ex:K24}.
The variety  $V_{K_{3,3}}$ has three components, of codimension 9 and degrees 10, 10, 492. The components of degree 10 are $V_{123456}$ and $V_{123456}^\ast$. The remaining component is 
the realization~$W_{K_{3,3}}$.
\end{example}

\begin{example}[$\ell = 7$, triangle-free]
Among the $59$ graphs, $52$ are irreducible by Theorem \ref{thm:irred}.
There are $5$ graphs $G$ which are (2,3)-sparse and CS, but not SCS, hence $V_G$ is a reducible complete intersection by Theorem \ref{thm:CI}.
Two of them come from adding a pendant edge to $K_{2,4}$. 
In both cases, $V_G$ has $2$ components of codimension $9$ and degrees $480$, $32$. 
The next graph is $K_{2,4}$ with a new vertex connected to two
bivalent vertices.
Here, $V_G$ has 2 components of codimension $10$ and degrees $960$, $64$. 
Two more graphs come from $K_{3,3}$ by adding a pendant edge or triangle.
The variety $V_G$ has $3$ components of codimension $10$ and degrees $20, 20, 984$
resp.~$6$ components of codimension $11$ and degrees $12, 12, 28, 28, 128, 1840$.
The two graphs which are reducible and not a complete intersection  are $K_{2,5}$ and $K_{3,4}$.
Example \ref{ex:K25} featured  $K_{2,5}$. The graph $K_{3,4}$ has $21$ components.
In addition to the coplanar and concurrent loci of codimension 11, we discovered
$19$ components of codimension $12$.
\end{example}

\begin{example}[$\ell = 8$, triangle-free]\label{eg:K44_1edge} Among the $267$ graphs, $266$ are realizable.
The unique  graph $G$ with $W_G = \emptyset$ is $K_{4,4}$ with one edge removed. 
Its variety $V_G$ has two components of codim $13$, four of codim $14$, and $57$ components of codim $15$.
 In each of these components, there is a non-edge $ij \notin G$ satisfying
(\ref{eq:tolerance}) on every witness set we computed.  
This is the first triangle-free incidence theorem we found.
For $\ell \leq 7$,   every triangle-free graph is realizable.
\end{example}

Complete bipartite graphs play a prominent role
in our discussion of triangle-free graphs.
We record the following fact about their realizations.
The proof of Proposition \ref{prop:rulings} is omitted.

\begin{proposition} \label{prop:rulings}
Let $G = K_{a,b}$ where $a,b\geq 3$.
Then $Y_G$ is irreducible, and it parametrizes
$a+b$ lines on some quadric surface, with $a$ lines in one ruling
and $b$ lines in the other ruling.
\end{proposition}

We next turn to connected graphs on $\ell = 6, 7$ vertices,
where triangles are now allowed.

\begin{example}[$\ell = 6$]
The smallest incidence theorems involve six lines.
Among the $112$ connected graphs, nine are not realizable. One such graph is the wheel graph $W_6$. Its incidence
variety $V_{W_6}$ has two components of codim $9$ (degrees $14, 14$) 
and $20$ components of codimension $10$. Among these
$20$ components, ten have degree $26$ and ten have degree $58$.
\end{example}

\begin{figure}[h]
	\centering
	\includegraphics[scale = .38]{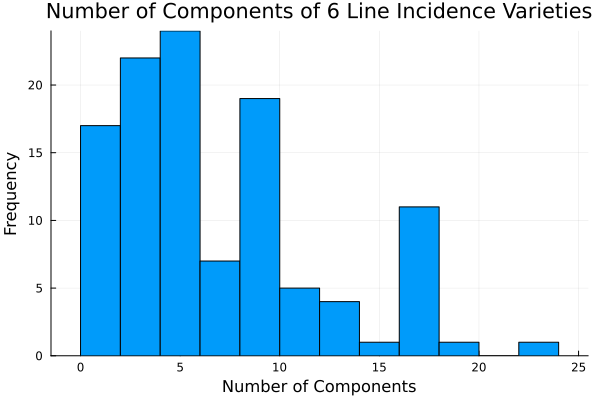} \includegraphics[scale = .38]{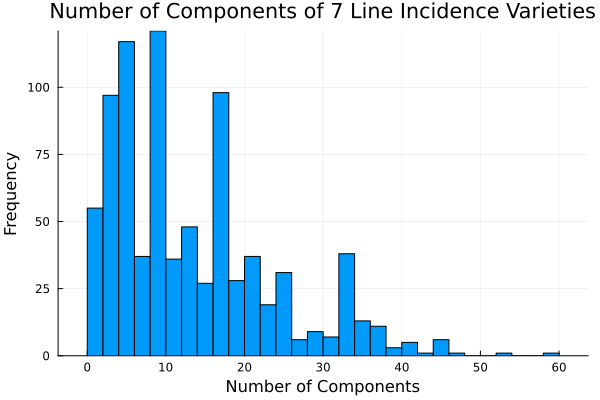}
	\caption{Histogram for the number of components of $V_G$ for connected graphs with $\ell=6,7$.}
	\label{figure:6-node-comps}
\end{figure}

The wheel graph $W_6$ has the maximum number of components, namely $22$, among all connected graphs 
with $\ell=6$. It is the only graph to achieve this maximum. The histograms  in Figure \ref{figure:6-node-comps} 
display the number of components for all incidence varieties
for $\ell \in \{6,7\}$.

Theorem \ref{thm:wheel} offers a general result on 
the irreducible components of the wheel graphs.
The wheel graph $W_7$ has $52$
 irreducible components, but Figure \ref{figure:6-node-comps}  shows
 that there is another graph with $\ell=7$ vertices which has more than $52$ components.
 We now present that winner.

\begin{proposition} \label{prop:winner7}
Let $G$ be the complete graph $K_7$ with a $7$-cycle removed, so
$|G| = 14$. Then $V_G$ has $58$ irreducible components,
which is the maximum among all graphs with $\ell= 7$.
The graph $G$ represents an incidence theorem for lines in $3$-space, i.e.~we have $W_G = \emptyset$.
\end{proposition}

\begin{proof}[Discussion]
The $17$-dimensional varieties $V_{[7]}$ and $V_{[7]}^*$
are  components of $V_G$. The other $56 = 4 \times 2 \times 7$ components
have dimension $\leq 15$.
There are four classes, each of size $2 \times 7$,
from the involution in Proposition \ref{prop:hodge}
plus cyclic symmetry. We describe each class
by the non-edges which vanish on that component; see~(\ref{eq:tolerance}).
We label the $7$ non-edges  by $\{i,i+1\}$.

The unique class of dimension $15$
is given by five consecutive non-edges $12,23,34,45,56$.
It has degree $38$. The other three classes have dimension $14$.
One of these classes has degree $58$. It is given by three consecutive non-edges $12,23,34$.
Another class has degree $124$, and it is given by three non-edges, like $12,23,56$,
where only one pair is adjacent. Finally, there is a class of 
degree $266$. Here only one non-edge enters the prime ideal of that component.
\end{proof}

All our data are posted on our supplementary materials website.
See (\ref{eq:zenodo}) for the URL.

\section{Towards Physics}
\label{sec9}

This project started from a discussion about
scattering amplitudes in particle physics.
In this section we explain the set-up to algebraically-minded
readers who are not familiar
with theoretical physics. It serves as a warm welcome to
our forthcoming physics paper \cite{HMPS}.
 
The determinant of a $2 \times 2$ matrix
 is a quadratic form on the vector space $\RR^4$.
This quadratic form has signature $(2,2)$.
To emphasize this, we use the notation $\RR^{2,2}$ for the space $\RR^4$.
The $(2,2)$ signature was made manifest in (\ref{eq:matricesST2}),
which was our point of entry to rigidity theory.
Viewing elements of $\RR^{2,2}$ as  point pairs allowed us to harness
the Geiringer-Laman Theorem.
The Grassmannian ${\rm Gr}(2,4)$ serves as the natural compactification (Lemma~\ref{lem:pointpairs}).

The setting in physics is similar, but the role
of the space $\RR^{2,2}$ is now played by Minkowski spacetime $\RR^{1,3}$.
We need to replace the signature $(2,2)$ with the Lorentzian signature $(1,3)$.
This is accomplished by a coordinate change which uses $i=\sqrt{-1}$.
The keyword for this transformation is {\em Wick rotation}.
Namely, we replace the $2 \times 4$ matrices ${\bf A}$ and ${\bf B}$ in (\ref{eq:ABCDE}) with
\begin{equation}
\label{eq:matricesPQ} 
{\bf P} \,\, = \,\, \begin{pmatrix}
1 & 0 & p_0-p_3 & p_1 + i p_2 \\
0 & 1 & p_1- i p_2 & p_0 + p_3 \end{pmatrix}
\quad {\rm and} \quad
 {\bf Q} \,\, = \,\, \begin{pmatrix}
1 & 0 & q_0-q_3 & q_1 + i q_2 \\
0 & 1 & q_1- i q_2 & q_0 + q_3 \end{pmatrix}.
\end{equation}
The vectors  $p = (p_0,p_1,p_2,p_3)$ 
and $q = (q_0,q_1,q_2,q_3)$ are elements in  $\RR^{1,3}$.
They are the {\em momentum vectors} of two elementary particles.
Their four coordinates encode the energies and momenta of the particles.
The $(1,3)$ signature is made manifest by the following identity
\begin{equation}
\label{eq:PdotQ}
\widetilde{PQ} \,\,  =\,\, {\rm det} \begin{small}
\begin{pmatrix} {\bf P} \\ {\bf Q} \end{pmatrix} \end{small} \,\, = \,\,
 (p_0 - q_0)^2 \, - \,
 (p_1 - q_1)^2 \, - \,(p_2 - q_2)^2 \, - \,(p_3 - q_3)^2 .
 \end{equation}
 This is the physics analogue to our earlier equations (\ref{eq:AdotB}) and (\ref{eq:matricesST2}).
 As before,  the vectors $p$ and $q$ in $\RR^{1,3}$ represent lines $P$ and $Q$  in $\PP^3$.
 These lines intersect if and only if (\ref{eq:PdotQ}) vanishes.
 
 An important object in theoretical physics is the {\em lightcone}
$\{ p_1^2 + p_2^2 + p_3^2 \leq p_0^2\}$ in $\RR^{1,3}$.
The ``universal speed limit'' stipulates that all particles must stay inside the lightcone.
Two particles $p$ and $q$ are {\em causally dependent} if $p-q$ is in the lightcone.
Otherwise they are {\em causally independent}. The equation $\widetilde{PQ} = 0$
delineates the boundary between dependence and independence. If it holds, then we say that the particles $p$ and $q$ are {\em light-like separated}.
We summarize the discussion above in the following statement which mirrors Lemma \ref{lem:pointpairs}. 

\begin{lemma} \label{lem:particlepairs}
The Grassmannian ${\rm Gr}(2,4)$ is a natural compactification of spacetime $\RR^{1,3}$.
Two particles are light-like separated if and only if the corresponding lines $P$ and $Q$ intersect.
\end{lemma}

In physics, this corresponds to a change of variables to \textit{momentum twistors}~\cite{Hodges}.
The boundary of the lightcone is the threefold defined by the 
quadric $p_0^2-p_1^2 - p_2^2 - p_3^2 $.
This quadric is the rightmost $2 \times 2$ minor of the matrix  ${\bf P}$.
This minor is a Pl\"ucker coordinate on the Grassmannian ${\rm Gr}(2,4)$.
Under the compactification in Lemma \ref{lem:particlepairs},
this has the following geometric interpretation:
the boundary of the lightcone is a Schubert divisor in ${\rm Gr}(2,4)$.

\begin{corollary} \label{cor:particles}
The incidence variety $V_G$ parametrizes configurations of $\ell$ particles in $\RR^{1,3}$,
where  light-like separation is prescribed for pairs of particles that are edges in the graph $G$.
\end{corollary}

We now take a step back and take a quick look at the larger scientific landscape.
Particle physics studies fundamental interactions among elementary particles in nature. A \emph{scattering experiment} consists in making $n$ particles with momenta $P,Q, \ldots \in \mathbb{R}^{1,3}$ collide in a particle accelerator, such as the Large Hadron Collider at CERN. By repeating the experiment many times, experimental physicists measure the joint probability for different outcomes as a function of the momenta. This probability can be obtained as the square modulus of its probability amplitude, called \emph{scattering amplitude}. Theoretical physicists use the framework of Quantum Field Theory (QFT) to compute such scattering amplitudes, both to compare with scattering experiments and to advance the mathematical understanding of QFT itself. 

Scattering amplitudes are computed by summing \emph{Feynman integrals} \cite{Wein}. These are functions of the momenta $P,Q,\ldots \in \mathbb{R}^{1,3}$ and are expressed as integrals of rational functions in $4\ell$ variables.
Their arguments are $\ell$ \emph{loop momenta} $K,L,\ldots \in \mathbb{R}^{1,3}$. That rational function
in $4 \ell$ variables is singular when some linear combinations of momenta and loop momenta is light-like separated. 
This is where  Lemma \ref{lem:particlepairs} comes in.
In the geometric approach we pursue,
the $n $ momenta $P,Q,\ldots $ and the $\ell$ loop momenta $K,L,\ldots$ are
lines in $3$-space, given by points in the Grassmannian ${\rm Gr}(2,4)$.
Light-like separation means that two lines intersect.

The rational functions to be integrated have denominators that are products of quadrics like 
$\widetilde{PQ}$ in (\ref{eq:PdotQ}).
The relation between the singular locus of these integrands and the singularities of Feynman integrals and scattering amplitudes
is the subject of the \emph{Landau analysis}.
 Recent work in~\cite{FMT} connects Landau analysis to computational algebraic geometry.
However, their geometry  looks disturbingly different from ours. This is a feature, not a bug.
 By  \cite[Section 2.5]{Wein},  there many different ways to write a  Feynman integral.
 Ours closely relates to the momentum representation, while Fevola et al.~\cite{FMT} use the Lee-Pomeransky representation.
 
The article \cite{HMPS} will develop \emph{Landau analysis in the Grassmannian}.
This is about functions on varieties of lines in $3$-space.
For a physical theory known as \emph{planar $\mathcal{N}=4$ super Yang-Mills} (sYM), one can write the sum over all Feynman integrals  for fixed $\ell$ as a single integral over a rational function in $\ell$ loop momenta. When expressed in terms of lines in 
$\PP^3$, the rational function is conjectured to be the \emph{canonical form} 
of a positive geometry called \emph{amplituhedron} \cite{AT}. Boundaries of amplituhedra 
are closely related to line incidence varieties. 
 Therefore our work is also connected with the fields of total positivity and positive geometry.

Every graph $G$ and every $u \in \NN^d$
specifies a Schubert problem. This was
described in Section \ref{sec3}.  Given
$n = 4\ell-c$ lines, we seek $\ell$ lines which intersect
each other and the given lines according to  $G_u$.
The number of solutions is the LS degree. The solution is
an algebraic function of the data. Usually,
these cannot be written in radicals when
the LS degree is~$\geq 5$.

The graphs $G_u$ are closely related to Feynman diagrams~\cite{Wein}.
According to Corollary~\ref{cor:particles}, edges in $G_u$ indicate which pairs
of momenta are light-like separated. Variants of the incidence variety $V_{G_u}$
can be found in the physics literature. The recent article \cite{HMSV} by
Hannesdottir et al.~uses the term {\em on-shell space} and attributes the concept
to work of Pham in the 1960s. The projection in
\cite[equation (5.4)]{HMSV} corresponds to our map  $\psi_u$ in Section~\ref{sec3}.

In physics, one is interested in special loci in the
image space ${\rm Gr}(2,4)^d$ of the map $\psi_u$.
This has two aspects. First, there is the {\em branch locus}. This consists
of all points in ${\rm Gr}(2,4)^d$ over which the fiber of $\psi_u$
does not consist of LS degree many distinct points.
The equation of this branch locus is called the {\em LS discriminant}. Second, we can fix 
a graph $H_u$ with vertex set~$[d]$ and require 
our $d$ external lines to lie in $V_{H_u} \subset {\rm Gr}(2,4)^d$.
The natural map
$\varphi_u: V_{G \cup H_u} \rightarrow V_{H_u}$
projects configurations of $\ell+d$ lines onto configurations of $d$ lines.
This is a variant of the map $\psi_u$, but now the 
Schubert problem has become easier. We can hope to 
 write down  LS degree many 
{\em rational formulas} in the $d$ given lines for the
$\ell$ desired lines. 

\begin{example}[Triangle revisited]
Let $\ell = 3$, $G = K_3$, $d=9$ and $u = (3,3,3)$.
We discuss the Landau analysis undertaken by
Bourjaily et al.~in \cite[Section III]{BVH}. Their pictures show
$\ell = 3$ pentagons,
representing loop momenta.
That Landau diagram is the planar dual to our graph $G_u$.
The legs $1,2,\ldots,11$ of their diagram specify a graph $H$ 
for the $d=9$ external momenta, and their figures in
\cite[eqns (17)--(22)]{BVH} describe
irreducible components of the variety $V_{G \cup H_u}$.
That variety is the on-shell space referred to  in the title of \cite[Section III(a)]{BVH}.

We saw in Example \ref{ex:three} that $V_G$ has two irreducible components.
Example \ref{ex:K3again} writes the Schubert problem on each component
as the {\em intersection of three quadrics} \cite[Section IV(c)]{BVH}.
The $8$ solutions form a  {\em Cayley octad} in $\PP^3$.
The LS discriminant is a huge polynomial in the $54$ Pl\"ucker coordinates
of the $d=9$ lines which vanishes when two of the $8$ solutions come together.
Properties of this polynomial and of other LS discriminants will be studied in \cite{HMPS}.
                                                                                           
Turning to the second aspect, let $H$ be the $9$-cycle,
where the three neighbors of each vertex of $G$ are a $3$-chain.
The external momenta are the lines $x_i x_{i+1}$ spanned by 
adjacent pairs in a cycle of generic points
$x_1,x_2,\ldots,x_9 \in \PP^3$. The fiber $\varphi_u$
consists of eight triples of rational lines.
The first triple is
$(x_2 z, \,x_5 z,\, x_8 z)$, where  $z$ is the intersection of the planes
$\underline{x_2} x_3 x_4$, $\underline{x_5} x_6 x_7$ and $\underline{x_8} x_9 x_1$.
For the other $7$ triples, we can replace (or not) the
first plane by $\underline{x_3} x_1 x_2$, the
second plane by $\underline{x_6} x_4 x_5$, and the
third plane by $\underline{x_9} x_7 x_8$.
In each of these $8$ cases, the three lines are spanned by the intersection point $z$
with the underlined $x_i$'s.
These rational formulas for the $8$ leading singularities
are similar to those in \cite[Section III(c)]{BVH}.
\end{example}

Cyclicity of external momenta is a hallmark of scattering amplitudes for planar $\mathcal{N} = 4$ sYM.
Whenever the graph $G_u$ is planar, physicists  assume that the $d$ external lines are spanned by consecutive columns
in a totally positive  $4\times 2d$ matrix. The important {\em reality conjecture} in \cite{HMPS}
states that, under that positivity hypothesis, all leading singularities have real coordinates. This would ensure that
the LS discriminant is positive on the positive Grassmannian ${\rm Gr}_+(4,2d)$.
LS discriminants are also important building blocks for the \emph{symbols} \cite{Golden_2014} of scattering amplitudes and they bear connections with cluster algebras.  

\smallskip

The present paper lays the foundation for Landau analysis in the Grassmannian \cite{HMPS}.
But, our results in Sections \ref{sec2}--\ref{sec6} are more widely applicable, well beyond
particle physics.
The incidence variety $V_G$ is  for anyone
who encounters configurations of lines in $3$-space. 
Among the many pertinent threads in mathematics,
rigidity theory is just the tip of the~iceberg.

\bigskip \medskip

\noindent {\bf Acknowledgement}:
We thank Jacob Bourjaily
and Cristian Vergu for valuable feedback during the development of this project, and the organizers of the UNIVERSE+ Annual Meeting at Ringberg Castle, where this work originated. 
 Our research was supported by
the European Research Council through the synergy grant UNIVERSE+, 101118787.
$\!\!$ \begin{scriptsize}Views~and~opinions expressed
are however those of the authors only and do not necessarily reflect those of the European Union or the 
European
Research Council Executive Agency. Neither the European Union nor the granting authority
can be held responsible for them.
\end{scriptsize}

\bigskip 

\noindent
\footnotesize {\bf Authors' addresses:}
\smallskip

\noindent Ben Hollering, MPI MiS Leipzig, Germany \hfill {\tt benjamin.hollering@mis.mpg.de}

\noindent Elia Mazzucchelli, MPI Physics, Garching, Germany \hfill {\tt eliam@mpp.mpg.de}

\noindent Matteo Parisi, MPI Physics, Garching, Germany and OIST, Okinawa, Japan \hfill {\tt matteo.parisi@oist.jp}

\noindent Bernd Sturmfels, MPI MiS Leipzig, Germany \hfill {\tt bernd@mis.mpg.de}

\end{document}